\documentclass{article}

\usepackage{amsmath, amssymb, amsthm}
\usepackage{mathtools}
\usepackage{mathrsfs}
\usepackage{bm}
\usepackage{extarrows}
\usepackage{geometry}
\usepackage{float}
\usepackage{indentfirst}
\usepackage{anyfontsize}
\usepackage{booktabs,multirow}
\usepackage{multicol}
\usepackage{enumitem}
\usepackage{appendix}

\usepackage[dvipsnames]{xcolor}
\usepackage{graphicx}
\graphicspath{
    {./figure/}{./figures/}{./image/}{./images/}{./graphic/}
    {./graphics/}{./picture/}{./pictures/}
}

\usepackage{subcaption}
\usepackage{epstopdf}

\usepackage{algorithm}
\usepackage{algorithmic}
\usepackage{hyperref}

\usepackage{thmtools}
\declaretheorem[style=plain, name=Theorem, numbered=yes]{theorem}
\declaretheorem[style=plain, name=Proposition, numbered=yes, sibling=theorem]{proposition}

\declaretheorem[style=plain, name=Lemma, numbered=yes, sibling=theorem]{lemma}

\declaretheorem[style=remark, name=Remark, numbered=yes]{remark}
\declaretheorem[style=remark, name=Note, numbered=no]{note*}

\DeclarePairedDelimiter{\inner}{\langle}{\rangle}
\DeclarePairedDelimiter{\norm}{\lVert}{\rVert}

\DeclareMathOperator{\diag}{diag}

\DeclareMathOperator{\qr}{qr}

\DeclareMathOperator{\spanop}{span}

\newenvironment{keywords}{\par\vskip 2ex\textbf{Keywords:}\enspace}{\par}

\title{An Asymptotic-Preserving Dynamical Low-Rank Semi-Lagrangian Method for Multiscale Linear Kinetic Transport Equations}

\author{
    Shun Li\thanks{School of Mathematical Sciences, University of Science and Technology of China. Email: {\tt lishun@mail.ustc.edu.cn}.} \and
    Yan Jiang\thanks{School of Mathematical Sciences, University of Science and Technology of China. Email: {\tt jiangy@ustc.edu.cn}. Research is partially supported by NSFC grant No.12271499.} \and
    Mengping Zhang\thanks{School of Mathematical Sciences, University of Science and Technology of China \& State Key Laboratory of Cognitive Intelligence. Email: {\tt mpzhang@ustc.edu.cn}.}
    \and
    Tao Xiong\thanks{School of Mathematical Sciences, University of Science and Technology of China \& State Key Laboratory of Cognitive Intelligence. Email: {\tt taoxiong@ustc.edu.cn}. Research is partially supported by National Key R\&D Program of China No. 2022YFA1004500, NSFC grant No.12571443.}
}

\date{}

\begin{document}

\maketitle

\begin{abstract}
    In this paper, we develop an asymptotic-preserving (AP) dynamical low-rank semi-Lagrangian method for multiscale linear kinetic transport equations.
    The method combines the large-time-step capability of semi-Lagrangian discretizations with the storage and cost reduction provided by low-rank representations.
    The proposed scheme couples an approximate macroscopic density update with the basis update Galerkin integrator for the kinetic distribution.
    To retain the reduced complexity in the semi-Lagrangian flux evaluation, the flux derivative is computed through a sampled angular quadrature strategy.
    We establish an unconditional stability analysis of the full-quadrature low-rank scheme in the constant-coefficient case.
    The error induced by angular sampling in the flux derivative is quantified.
    The resulting scheme is shown to be AP in the diffusive limit.
    Numerical experiments, including high-dimensional test cases, demonstrate that the proposed method is AP, stable under large time steps, and computationally efficient across kinetic and diffusive regimes.
\end{abstract}


\section{Introduction}

Kinetic transport equations are fundamental models in radiative transfer, neutron transport, rarefied gas dynamics, and kinetic gas theory \cite{case1967,Chandrasekhar1960,lewis1983computational,Cercignani1988,DimarcoPareschi2014}.
In the diffusive scaling, the Knudsen number $\varepsilon$ may be small, so the kinetic model approaches a macroscopic diffusion equation and direct discretizations face both multiscale stiffness and high-dimensional phase-space cost.
Asymptotic-preserving (AP) schemes address the stiffness by capturing the correct macroscopic limit as $\varepsilon\to0$ without explicitly resolving the small scale \cite{jin1999efficient,jin2022asymptotic}.
For linear kinetic equations, such schemes have been developed using micro--macro, implicit-explicit (IMEX) discretizations, discontinuous Galerkin (DG), and semi-Lagrangian discretizations \cite{lemou2008new,liu2010analysis,boscarino2013implicit,JLQX15,ding2023accuracy,peng2021asymptotic,zhang2023asymptotic,cai2024asymptotic}.

Low-rank approximation provides a complementary way to reduce the cost and storage of high-dimensional kinetic discretizations \cite{bachmayr2023low,einkemmer2025review}.
Dynamical low-rank (DLR) approximation evolves a low-rank representation by projecting the governing equation onto the tangent space of a fixed-rank matrix manifold \cite{koch2007dynamical}.
Projector-splitting methods, basis update Galerkin (BUG) integrators, and rank-adaptive variants provide practical and robust time integrators for such low-rank evolutions \cite{lubich2014projector,kieri2016discretized,ceruti2022unconventional,ceruti2022rank}.
These ideas have been applied to Vlasov, Bhatnagar--Gross--Krook (BGK), neutron transport, and radiative-transfer models \cite{einkemmer2018lowrank,einkemmer2021efficient,peng2020low,peng2023sweep,kusch2022low,kusch2023robust}.
As an alternative approach, the step-and-truncation (SAT) method builds a low-rank solution from a traditional full-rank discretization, which updates the solution and then truncates small singular modes at each time step \cite{guo2024conservative,guo2024local,sands2025high}.
Recent sweep-based low-rank methods also form an efficient class of solvers for steady-state radiative-transfer equations \cite{peng2023sweep,guo2025inexact,guo2026highly,haut2026efficient}.

Several recent developments have sought to improve low-rank kinetic solvers not only in storage efficiency but also in time stepping and online evaluation.
One line combines SL transport with the SAT method.
Beginning with SL tensor-train solvers for Vlasov equations \cite{Kormann2015}, recent SL adaptive-rank (SLAR) approaches use pseudo-skeleton approximations \cite{tyrtyshnikov1995pseudo} to retain the large-time-step advantage of SL discretizations while controlling ranks and costs \cite{zheng2025semi,zheng2025highDimensionalSLAR}.
Another line addresses the multiscale behavior of low-rank transport solvers.
AP-DLR methods based on macro--micro formulations have been developed for multiscale linear transport equations \cite{einkemmerHuWang2021apDLR}, diffusion-limit error analysis is available for DLR integrators for linear Boltzmann equations \cite{ding2021dynamical},
and energy-stable AP-DLR or BUG-type formulations further address stability and AP behavior of the low-rank microscopic component \cite{einkemmerHuKusch2024energyStable,frank2025asymptotic}.
Temporal-stability-enhanced low-rank discretizations provide routes toward large time steps in the diffusive regime \cite{liJiangZhangXiong2026temporalStability}, while the Galerkin alternating projection method combines a fully implicit discretization with modified low-rank integrators to obtain unconditional stability and AP behavior \cite{ceruti2025galerkin}.
A third ingredient is hyper-reduction by interpolation or sampling.
In reduced-order modeling, the discrete empirical interpolation method (DEIM) \cite{chaturantabut2010nonlinear}, QR-based DEIM (QDEIM) \cite{drmac2016new}, and empirical cubature \cite{an2008optimizing,hernandez2017dimensional} are standard tools for avoiding full-dimensional online evaluations.
Related ideas have recently been used in low-rank kinetic and radiative-transfer solvers, including DEIM-type tensor-cross algorithms \cite{Ghahremani2024}, interpolatory DLR methods \cite{dektor2024interpolatoryBGK}, selected-angle $S_N$-like DLR formulations \cite{haut2026efficient,guo2026highly}, and greedy angular sampling strategies \cite{sandsQiuHayesZheng2026greedyBGK}.

In this work, we develop an AP dynamical low-rank semi-Lagrangian method for multiscale linear kinetic transport equations.
Starting from the AP FD-SL scheme of \cite{zhang2023asymptotic}, the method introduces a weighted low-rank representation of the kinetic distribution and advances it by a BUG integrator, while the macroscopic density is updated through an approximate SL flux derivative.
The main computational difficulty is that this flux derivative involves direction-dependent characteristic tracing and angular integration.
A direct evaluation over all angular directions would largely destroy the reduced complexity gained from the low-rank representation.
To retain the low-rank complexity, we evaluate it using QDEIM-selected angular directions together with effective quadrature weights that enforce exactness on a target angular space.
For the constant-coefficient problem, we establish an unconditional stability analysis of the full-quadrature low-rank scheme.
We also quantify the error induced by angular sampling in the flux derivative and show that both the full-quadrature and sampled schemes are AP in the diffusive limit.
Numerical experiments verify the AP property, large-time-step stability, and computational efficiency of the proposed method in both kinetic and diffusive regimes.

For the remainder,
Section~\ref{sec:background} recalls the SL reformulation and DLR framework.
Section~\ref{sec:numerical-method} gives the full-rank and low-rank schemes, including the multidimensional extension and complexity.
Section~\ref{sec:theoretical-analysis} establishes constant-coefficient energy stability, estimates the flux-derivative sampling error, and proves AP inheritance.
Section~\ref{sec:numerical-results} reports numerical tests, and Section~\ref{sec:conclusion} concludes.

\section{Background}\label{sec:background}

\subsection{Model problem}
We consider the multiscale linear kinetic transport equation
\begin{equation}\label{eq:model-general}
    \partial_t f + \frac{1}{\varepsilon}\bm{\Omega}\cdot\nabla_{\bm{x}} f
    =
    \frac{\sigma^s}{\varepsilon^2}\bigl(\inner*{f}_{\bm{\Omega}}-f\bigr) - \sigma^a f + \Phi.
\end{equation}
Here $f = f(\bm{x},\bm{\Omega},t)$ denotes the particle density at time $t \ge 0$, position $\bm{x}\in D\subset\mathbb{R}^{3}$, and angular direction $\bm{\Omega}\in\mathbb{S}^{2}$.
The source term is denoted by $\Phi = \Phi(\bm{x},t)$.
The normalized angular average is defined by
\[
    \inner*{f}_{\bm{\Omega}} = \frac{1}{4\pi} \int_{\mathbb{S}^2} f\,\mathrm{d}\,\bm{\Omega}.
\]
The macroscopic density is given by $\rho = \inner*{f}_{\bm{\Omega}}$.
The scattering and absorption coefficients satisfy $\sigma^s(\bm{x}) \geq \sigma^s_{\min} > 0$ and $\sigma^a(\bm{x}) \geq 0$, respectively.
The Knudsen number $\varepsilon > 0$ denotes the ratio of the mean free path of particles to the characteristic length.
Taking the angular average of \eqref{eq:model-general} gives
\[
    \partial_t\rho
    +
    \frac{1}{\varepsilon}
    \nabla_{\bm{x}}\cdot\inner*{\bm{\Omega} (f-\rho)}_{\bm{\Omega}}
    =
    -\sigma^a(\bm{x})\rho+\Phi(\bm{x},t).
\]
In the diffusive limit $\varepsilon\to0$, the kinetic solution relaxes toward its angular average, and the limiting density satisfies
\[
    \partial_t \rho
    =
    \nabla_{\bm{x}}\cdot\left(\frac{1}{3\sigma^s(\bm{x})}
    \nabla_{\bm{x}}\rho\right)
    -
    \sigma^a(\bm{x})\rho
    +
    \Phi(\bm{x},t).
\]
The factor $1/3$ comes from the normalized spherical average $\inner*{\Omega_i\Omega_j}_{\bm{\Omega}}=\delta_{ij}/3$.

We describe the method in the 1D1V case to keep the notation simple.
The extension to 2D2V and 3D2V follows the same construction.
In the 1D1V case, $v\in[-1,1]$ and \eqref{eq:model-general} reduces to
\[
    f_t + \frac{v}{\varepsilon} f_x = \frac{\sigma^s}{\varepsilon^2}\bigl(\rho-f\bigr) -\sigma^a f+\Phi,
    \quad
    \rho(x,t)=\inner*{f} = \frac12 \int_{-1}^1 f(x,v,t)\,\mathrm{d} v.
\]
The averaged equation is
\[
    \rho_t+\frac1\varepsilon\partial_x\inner*{v(f-\rho)}
    =
    -\sigma^a\rho+\Phi.
\]

\subsection{Semi-Lagrangian reformulation}

We recall the semi-Lagrangian AP reformulation \cite{zhang2023asymptotic,cai2024asymptotic}, which derives an approximate evolution equation for the macroscopic density from the formal solution of the distribution function.

Fix a point $(x_\ast,t_\ast)$ and a velocity $v$. The characteristic passing through $(x_\ast,t_\ast)$ is
\[
    X(t)=x_\ast-\frac{v}{\varepsilon}(t_\ast-t).
\]
Along this curve, the kinetic equation becomes
\[
    \frac{\mathrm{d}}{\mathrm{d}t}f(X(t),v,t)
    + \mu f(X(t),v,t)
    =
    \frac{\sigma^s}{\varepsilon^2}\rho(X(t),t)
    +
    \Phi(X(t),t),
\]
where $\mu = \sigma^s / \varepsilon^2 + \sigma^a$.
Using the exponential integrating factor and integrating from $t_n$ to $t_\ast$ gives
    {\small\begin{equation}\label{eq:fdsl-integrating-factor-solution}
            f(x_\ast,v,t_\ast)
            =
            e^{-\mu(t_\ast-t_n)} f(X(t_n),v,t_n)
            +
            \int_{t_n}^{t_\ast} e^{-\mu(t_\ast-\tau)} \left( \frac{\sigma^s}{\varepsilon^2}\rho(X(\tau),\tau) + \Phi(X(\tau),\tau) \right) \,\mathrm{d}\tau,
        \end{equation}}
where the coefficients are frozen locally at $x_\ast$.
Since
\[
    \frac{\mathrm{d}}{\mathrm{d}\tau}\rho(X(\tau),\tau)
    =
    \partial_t\rho(X(\tau),\tau)
    +
    \frac{v}{\varepsilon}\partial_x\rho(X(\tau),\tau),
\]
integrating the density in \eqref{eq:fdsl-integrating-factor-solution} by parts gives
\begin{align*}
    f(x_\ast,v,t_\ast)
    ={}       &
    \frac{\sigma^s}{\mu\varepsilon^2}\rho(x_\ast,t_\ast)
    +
    e^{-\mu(t_\ast-t_n)}
    \left(
    f(X(t_n),v,t_n)
    -
    \frac{\sigma^s}{\mu\varepsilon^2}\rho(X(t_n),t_n)
    \right)
    \\
                & -
    \int_{t_n}^{t_\ast}
    \frac{\sigma^s}{\mu\varepsilon^2}
    e^{-\mu(t_\ast-\tau)}
    \left(
    \partial_t\rho(X(\tau),\tau)
    +
    \frac{v}{\varepsilon}\partial_x\rho(X(\tau),\tau)
    \right)
    \,\mathrm{d}\tau
    \\
                & +
    \int_{t_n}^{t_\ast}
    e^{-\mu(t_\ast-\tau)}
    \Phi(X(\tau),\tau)
    \,\mathrm{d}\tau
    \\
    \approx{} &
    \frac{\sigma^s}{\mu\varepsilon^2}\rho(x_\ast,t_\ast)
    +
    e^{-\mu(t_\ast-t_n)}
    \left(
    f(X(t_n),v,t_n)
    -
    \frac{\sigma^s}{\mu\varepsilon^2}\rho(X(t_n),t_n)
    \right)
    \\
                & -
    \left(
    \partial_t\rho(x_\ast,t_\ast)
    +
    \frac{v}{\varepsilon}\partial_x\rho(x_\ast,t_\ast)
    \right)
    \frac{\sigma^s}{\mu^2\varepsilon^2}
    \left(1-e^{-\mu(t_\ast-t_n)}\right)
    \\
                & +
    \frac{1-e^{-\mu(t_\ast-t_n)}}{\mu}
    \Phi(x_\ast,t_\ast).
\end{align*}
Here we freeze the slowly varying factors in the remaining integrals at $(x_\ast,t_\ast)$.
Taking the first velocity moment of $f-\rho$ yields
\begin{align*}
    \inner*{v(f-\rho)}(x_\ast,t_\ast)
    \approx{} &
    e^{-\mu(t_\ast-t_n)}
    \inner*{
        v(f-\rho)
        \left(
        x_\ast-\frac{v}{\varepsilon}(t_\ast-t_n),
        v,t_n
        \right)
    }             \\
              & -
    \frac{\sigma^s}{3\mu^2\varepsilon^3}
    \left(1-e^{-\mu(t_\ast-t_n)}\right)
    \partial_x\rho(x_\ast,t_\ast).
\end{align*}
Combining this approximation with the averaged equation gives the following approximate evolution equation for the macroscopic density:
\begin{equation}\label{eq:fdsl-approx-density-equation}
    \rho_t
    +
    \frac{\alpha_1}{\varepsilon}
    \partial_x
    \left(
    \inner*{
            v(f-\rho)
            \left(
            x-\frac{v}{\varepsilon}(t-t_n),
            v,t_n
            \right)
        }
    \right)
    -
    \partial_x
    \left(
    \frac{\alpha_2}{3\sigma^s}
    \partial_x\rho
    \right)
    =
    -\sigma^a\rho+\Phi,
\end{equation}
where
\[
    \alpha_1=e^{-\mu(t-t_n)},
    \qquad
    \alpha_2=
    \left(
    \frac{\sigma^s}{\mu\varepsilon^2}
    \right)^2
    \left(1-e^{-\mu(t-t_n)}\right).
\]
The first correction term in \eqref{eq:fdsl-approx-density-equation} is the transport part obtained by following characteristics, while the second correction term is the diffusion part.
As $\varepsilon\to0$, one has $\alpha_1/\varepsilon\to0$ and $\alpha_2\to1$, so \eqref{eq:fdsl-approx-density-equation} reduces to the diffusion limit
\[
    \rho_t
    -
    \partial_x
    \left(
    \frac{1}{3\sigma^s}
    \partial_x\rho
    \right)
    =
    -\sigma^a\rho+\Phi.
\]

\begin{remark}\label{rem:source-velocity-dependent}
    In some special cases, such as the manufactured-solution accuracy test, the constructed source term may depend on $v$.
    If $\Phi=\Phi(x,v,t)$, then the averaged source and the first source moment enter the evolution equation, and \eqref{eq:fdsl-approx-density-equation} is replaced by
    \begin{align*}
        \rho_t
        +
        \frac{\alpha_1}{\varepsilon}
        \partial_x
        \left(
        \inner*{v(f-\rho)}
        \left(x-\frac{v}{\varepsilon}(t-t_n),t_n\right)
        \right)
        -
        \partial_x\left(\frac{\alpha_2}{3\sigma^s}\partial_x\rho\right)
        ={} \\
        -\sigma^a\rho
        +\inner*{\Phi}
        -
        \frac{1}{\varepsilon}\partial_x
        \left(
        \frac{1-e^{-\mu(t-t_n)}}{\mu}
        \inner*{v\Phi}
        \right).
    \end{align*}
    The last term on the right-hand side vanishes in the diffusion limit.
\end{remark}

\subsection{Dynamical low-rank framework}

We briefly recall the dynamical low-rank (DLR) framework used for the microscopic update \cite{koch2007dynamical,lubich2014projector,kieri2016discretized}.
Let $A(t)\in\mathbb{R}^{m_1\times m_2}$ satisfy
\[
    \partial_tA(t)=\mathcal{F}(t,A(t)).
\]
DLR approximates $A(t)$ on the rank-$r$ matrix manifold $\mathcal{M}_r$ by
\[
    A(t)\approx Y(t)=X(t)S(t)V(t)^\top,
\]
where $X(t)\in\mathbb{R}^{m_1\times r}$ and $V(t)\in\mathbb{R}^{m_2\times r}$ have orthonormal columns, and $S(t)\in\mathbb{R}^{r\times r}$.
The evolution is obtained by projecting the vector field onto the tangent space of $\mathcal{M}_r$:
\begin{equation}\label{eq:dlr-tangent-projected-evolution}
    \partial_tY(t)=\mathcal{P}_{Y(t)}\mathcal{F}(t,Y(t)),
\end{equation}
where $\mathcal{P}_{Y}$ denotes the orthogonal projection onto $\mathcal{T}_{Y}(\mathcal{M}_r)$.
Let the singular value decomposition (SVD) of $Y \in \mathcal{M}_r$ be
$Y = X S V^\top$. The orthogonal projection $\mathcal{P}_{Y}$ is explicitly given by \cite[Lemma~4.1]{koch2007dynamical}
\[
    \mathcal{P}_Y(Z)
    =
    XX^\top Z+ZVV^\top-XX^\top ZVV^\top,
    \quad
    Z\in\mathbb{R}^{m_1\times m_2}.
\]

The BUG integrator gives a practical time discretization of \eqref{eq:dlr-tangent-projected-evolution} by evolving the factors through $K$-, $L$-, and $S$-steps \cite{ceruti2022unconventional}.
It preserves the prescribed rank during the update.
Rank-adaptive variants, such as the aBUG integrator \cite{ceruti2022rank}, add truncation or augmentation steps to select the rank according to a tolerance.
We use the fixed-rank BUG form in the method below; broader variants of low-rank solvers for kinetic equations are reviewed in \cite{einkemmer2025review}.

\section{Numerical Method}\label{sec:numerical-method}

\subsection{Full-Rank Scheme}

We first present an FD-SL scheme based on \cite{zhang2023asymptotic}, which serves as the full-rank reference scheme. We discretize the spatial interval $[x_L,x_R]$ by the uniform periodic grid
\[
    \Delta x=\frac{x_R-x_L}{N_x},\qquad
    x_i=x_L+(i-1)\Delta x,\quad i=1,\ldots,N_x,
\]
and discretize the angular variable by a quadrature rule $\{(v_j,w_j)\}_{j=1}^{N_v}$ with normalized weights, $\sum_{j=1}^{N_v} w_j=1$.
Define the following notation:
\[
    F^n = [f_1^n,\dots,f_{N_v}^n] \in \mathbb{R}^{N_x \times N_v},
    \quad  f_j^n=((f_{j}^n)_1,\ldots,(f_{j}^n)_{N_x})^\top,
\]
where $(f_{j}^n)_i \approx f(x_i, v_j, t^n)$.
Let $w=(w_1,\ldots,w_{N_v})^\top$ and $v=(v_1,\ldots,v_{N_v})^\top$.
The discrete density is given by $\rho^n=F^n w \in \mathbb{R}^{N_x}$.
We also denote by $\sigma^s,\sigma^a,\Phi^{n+1}\in\mathbb{R}^{N_x}$ the corresponding grid vectors, and define $\alpha_1,\alpha_2,\beta\in\mathbb{R}^{N_x}$ pointwise by
\[
    \mu_i=\frac{\sigma_i^s}{\varepsilon^2}+\sigma_i^a,
    \alpha_{1,i}=e^{-\mu_i\Delta t},
    \alpha_{2,i}
    =
    \left(\frac{\sigma_i^s}{\mu_i\varepsilon^2}\right)^2
    \left(1-e^{-\mu_i\Delta t}\right),
    \beta_i=\frac{\alpha_{2,i}}{3\sigma_i^s}.
\]
Define the upwind difference operator associated with $v_j$ by
\[
    D_x^{\mathrm{up},(j)}
    :=
    \begin{cases}
        D_x^-, & v_j\ge0, \\
        D_x^+, & v_j<0.
    \end{cases}
\]
Here $(D_x^- u)_i=(u_i-u_{i-1})/\Delta x$ and $(D_x^+ u)_i=(u_{i+1}-u_i)/\Delta x$ for any spatial grid function $u$.
The variable-coefficient diffusion operator is discretized by
\[
    D_x(\beta D_x u)_i
    =
    \frac{1}{\Delta x^2}\left( \beta_{i+1/2}(u_{i+1}-u_i) - \beta_{i-1/2}(u_i-u_{i-1}) \right),
    \quad
    \beta_{i+1/2}=\frac12(\beta_i+\beta_{i+1}).
\]
Boundary rows are modified according to the prescribed boundary condition.
Let $L_\beta \in \mathbb{R}^{N_x \times N_x}$ denote the matrix representation of $u\mapsto D_x(\beta D_x u)$.

The semi-Lagrangian flux derivative requires the spatial derivative of $f-\rho$ evaluated at the characteristic foot point $x_i-v_j\Delta t/\varepsilon$.
For each velocity $v_j$, we first apply the corresponding upwind difference operator to $f_j^n-\rho^n$, and then evaluate the resulting grid function at the foot point by linear interpolation.
Let $\mathcal{B}_j$ denote the linear interpolation operator. Under periodic boundary conditions,
\[
    \mathcal{B}_j=(1-\lambda_j)S_{m_j}+\lambda_j S_{m_j+1},
    \qquad
    \frac{v_j\Delta t}{\varepsilon\Delta x}=m_j+\lambda_j,\quad
    m_j\in\mathbb{Z},\quad 0\le\lambda_j<1,
\]
where $(S_m q)_i=q_{i-m}$ denotes the periodic grid shift.
For each velocity $v_j$, we define the backtracked derivative vector by
\begin{equation}\label{eq:fullrank-backtrack-stencil}
    b_j^n := \mathcal{B}_j \left(D_x^{\mathrm{up},(j)}(f_j^n-\rho^n)\right) \in \mathbb{R}^{N_x}.
\end{equation}
Componentwise,
\[
    (b_j^n)_i \approx \partial_x(f-\rho)\left(x_i-\frac{v_j}{\varepsilon}\Delta t,v_j,t_n\right).
\]
Collecting these vectors gives
\[
    B^n=[b_1^n,\ldots,b_{N_v}^n]\in\mathbb{R}^{N_x\times N_v}.
\]
The discrete backtracked macroscopic flux derivative is then assembled at each grid point $x_i$ as
\[
    \mathcal{J}_i^n = \sum_{j=1}^{N_v}w_jv_j (b_j^n)_i \approx \partial_x \inner*{v(f-\rho)\left(x_i-\frac{v}{\varepsilon}\Delta t,v,t_n\right)},
\]
or, equivalently,
\[
    \mathcal{J}^n = B^n(w\odot v)\in\mathbb{R}^{N_x}.
\]
Here $\odot$ denotes componentwise multiplication.

We now describe the full-rank FD-SL update.
At the beginning of each time step, we assemble $\mathcal{J}^n$ as described above.
The provisional density $\rho^{n+1,*}$ is then computed from the discretization of \eqref{eq:fdsl-approx-density-equation}:
\[
    \frac{\rho^{n+1,*}-\rho^n}{\Delta t} + \frac{1}{\varepsilon} \diag(\alpha_1) \mathcal{J}^n - L_\beta \rho^{n+1,*}
    = - \diag(\sigma^a) \rho^{n+1,*} +\Phi^{n+1},
\]
which can be rearranged into the macroscopic linear system
\begin{equation}\label{eq:macroscopic-linear-system}
    \left( I - \Delta t L_\beta + \Delta t \diag(\sigma^a) \right) \rho^{n+1,*}
    =
    \rho^n - \frac{\Delta t}{\varepsilon}\diag(\alpha_1)\mathcal{J}^n + \Delta t \Phi^{n+1}.
\end{equation}
After obtaining $\rho^{n+1,*}$, we substitute it into the microscopic equation and apply a backward Euler discretization in time.
Since $\rho^{n+1,*}$ is fixed, the update decouples over velocities. Each column $f_j^{n+1}$ satisfies
\[
    \frac{f^{n+1}_j-f^n_j}{\Delta t} + \frac{v_j}{\varepsilon} D_x^{\mathrm{up},(j)} f^{n+1}_j
    =
    \frac{1}{\varepsilon^2}\diag(\sigma^s) \bigl(\rho^{n+1,*}-f^{n+1}_j\bigr) - \diag(\sigma^a) f^{n+1}_j +\Phi^{n+1}.
\]
Thus the velocity-column solves are independent and can be carried out in parallel.
Finally, the macroscopic density is reset by angular averaging, $\rho^{n+1}=F^{n+1}w$,
so that it is consistent with the updated distribution.

\subsection{Low-Rank Scheme}

The low-rank scheme is built upon the full-rank FD-SL discretization.
It combines an energy-consistent weighted low-rank representation with the BUG integrator for the microscopic update, while the macroscopic semi-Lagrangian flux derivative is assembled efficiently by a sampled angular quadrature.
This sampled low-rank scheme is denoted by SL-DLR in what follows.
We also keep the corresponding full-quadrature low-rank scheme, denoted by SL-DLR(full), as a reference to separate the fixed-rank approximation error from the sampling error.

We introduce the weighted low-rank representation \cite{liJiangZhangXiong2026temporalStability}
\[
    Y^n = F^n M = X^n S^n (V^n)^\top,
\]
where $M=\diag(\sqrt{w_1},\dots,\sqrt{w_{N_v}})$, $X^n\in\mathbb R^{N_x\times r}$ and $V^n\in\mathbb R^{N_v\times r}$ have orthonormal columns,
and $S^n\in\mathbb R^{r\times r}$.
The macroscopic density is then recovered as
\[
    \rho^n=F^nw=X^nS^n(V^n)^\top M\bm{1}.
\]
We now describe one time step of the low-rank scheme from $t_n$ to $t_{n+1}$, assuming that $Y^n=X^nS^n(V^n)^\top$ and $\rho^n$ are given.

We first construct the semi-Lagrangian flux derivative in the macroscopic equation. This is the only stage where angular sampling is required.
At the continuous level, this term has the form
\[
    \partial_x \inner*{v(f - \rho)\left(x - v \Delta t / \varepsilon, v, t_n\right)}.
\]
Since the characteristic foot point depends on the angular direction, this backtracked flux derivative is not separable in the low-rank factors.
Hence, a low-rank discretization using the full angular quadrature for the flux derivative would still need to form the physical-space columns associated with all $N_v$ velocities and backtrack them before summation.

To overcome this difficulty, we use a sampled angular quadrature to reduce this cost while approximating the flux derivative.
At time $t_n$, assume that a sampled direction set $\mathcal{I}_n$ and the associated effective quadrature weights $\widetilde{w}^n$ have been specified:
\[
    \mathcal{I}_n=\{i_1,\ldots,i_{m_n}\},
    \,
    P_n=[e_{i_1},\ldots,e_{i_{m_n}}] \in \mathbb{R}^{N_v \times m_n},
    \,
    \widetilde{w}^n
    =[\widetilde{w}^n_{1},\ldots,\widetilde{w}^n_{m_n}]^\top
    \in\mathbb{R}^{m_n}.
\]
Here $m_n < N_v$, and $e_i$ denotes the $i$-th column of the identity matrix.
The construction of $P_n$ and $\widetilde{w}^n$ will be described in Subsection~\ref{sec:sampling}.
We assemble the sampled counterpart of the full-rank backtracked flux derivative:
\begin{equation}\label{eq:lr-sampled-flux}
    \widetilde{\mathcal{J}}^n
    =
    (B^n P_n)\bigl(\widetilde{w}^n \odot (P_n^\top v)\bigr)
    \in\mathbb{R}^{N_x}.
\end{equation}
The computation of $\widetilde{\mathcal{J}}^n$ does not require forming the full matrix $F^n$;
it only uses the sampled physical columns
\[
    F^nP_n
    =
    X^nS^n(V^n)^\top M^{-1}P_n
    \in\mathbb{R}^{N_x\times m_n}.
\]
Using this sampled flux derivative, we solve the corresponding macroscopic linear system, analogous to \eqref{eq:macroscopic-linear-system} in the full-rank scheme, to obtain $\rho^{n+1,*}$:
\begin{equation}\label{eq:sampled-macroscopic-linear-system}
    \left(I - \Delta t L_\beta + \Delta t \diag(\sigma^a)\right)\rho^{n+1,*}
    =
    \rho^n - \frac{\Delta t}{\varepsilon}\diag(\alpha_1) \widetilde{\mathcal{J}}^n + \Delta t \Phi^{n+1}.
\end{equation}

After the macroscopic solve, we update the distribution matrix on the low-rank manifold with $\rho^{n+1,*}$ fixed.
Define
\[
    Q=\diag(v_1,\cdots,v_{N_v}),
    \quad
    |Q|=\diag(|v_1|,\cdots,|v_{N_v}|),
    \quad
    Q^\pm=\frac12(Q\pm |Q|).
\]
Let $D_x^-$ and $D_x^+$ denote the backward and forward difference matrices, respectively.
For the weighted matrix $Y(t)=F(t)M$, the microscopic evolution is written as
    {\small\begin{align*}
            \partial_t Y
            ={} \mathcal{G}\bigl(Y;\rho^{n+1,*}\bigr)
            ={} &
            -\frac{1}{\varepsilon}
            \bigl(D_x^- Y Q^+ + D_x^+ Y  Q^-\bigr)
            \\
                & +
            \frac{1}{\varepsilon^2}
            \diag(\sigma^s)(\rho^{n+1,*} (M\bm{1})^\top - Y)
            -
            \diag(\sigma^a) Y
            +
            \Phi^{n+1} (M\bm{1})^\top.
        \end{align*}}
The BUG integrator is then applied to $\partial_tY=\mathcal{G}(Y;\rho^{n+1,*})$ with initial value $Y^n=X^nS^n(V^n)^\top$:
\begin{enumerate}
    \item \textbf{K-step:} For an initial condition $K^n = X^n S^n \in\mathbb{R}^{N_x \times r}$, we solve the following equation to obtain $K^{n+1}$
          \[
              \frac{K^{n+1}-K^n}{\Delta t}
              =
              \mathcal{G}\bigl(K^{n+1}(V^n)^\top;\rho^{n+1,*}\bigr)V^n.
          \]

    \item \textbf{L-step:}  For an initial condition $L^n = V^n (S^n)^\top \in \mathbb{R}^{N_v \times r}$, we solve the following equation to obtain $L^{n+1}$
          \[
              \frac{L^{n+1}-L^n}{\Delta t}
              =
              \mathcal{G}\bigl(X^n(L^{n+1})^\top;\rho^{n+1,*}\bigr)^\top X^n.
          \]

    \item Update basis: $[X^{n+1}, \sim] = \qr(K^{n+1})$, $[V^{n+1}, \sim] = \qr(L^{n+1})$.

    \item \textbf{S-step:} For an initial condition $\widetilde{S}^n=(X^{n+1})^\top X^n S^n (V^n)^\top V^{n+1}$, we solve the following equation to obtain $S^{n+1}$.
          \[
              \frac{S^{n+1}-\widetilde{S}^n}{\Delta t}
              =
              (X^{n+1})^\top
              \mathcal{G}\bigl(X^{n+1}S^{n+1}(V^{n+1})^\top;\rho^{n+1,*}\bigr)
              V^{n+1}.
          \]

    \item Update $Y^{n+1} = F^{n+1}M = X^{n+1}S^{n+1}(V^{n+1})^\top$.
\end{enumerate}

After obtaining $Y^{n+1}$, we further update $\rho^{n+1} = X^{n+1}S^{n+1}(V^{n+1})^\top M\bm{1}$, so that it is consistent with the updated distribution.

\subsection{Sampled Angular Quadrature}\label{sec:sampling}

We now describe the sampled angular quadrature used for the flux-derivative vector in \eqref{eq:lr-sampled-flux}.
The construction has two components: a set of sampled angular directions and the effective quadrature weights.
The sampled directions are selected from a target angular space by a QDEIM-type procedure \cite{drmac2016new}, while the effective weights are determined by exactness equations on the same space, in the spirit of empirical cubature rules \cite{an2008optimizing,hernandez2017dimensional}.
Thus, the resulting rule is exact on the prescribed target space and is used to approximate the backtracked semi-Lagrangian flux derivative without sweeping over all $N_v$ angular directions.

We first identify the angular vector to which the sampled quadrature is applied. For each spatial grid point $x_i$, define
\[
    h_i^n = \bigl(v_j(b_j^n)_i\bigr)_{j=1}^{N_v}\in\mathbb{R}^{N_v},
\]
where $b_{j}^n$ is defined in \eqref{eq:fullrank-backtrack-stencil}.
Then the full and sampled flux derivatives can be written as
\[
    \mathcal{J}_i^n=w^\top h_i^n,
    \qquad
    \widetilde{\mathcal{J}}_i^n=(\widetilde{w}^n)^\top P_n^\top h_i^n.
\]
The sampled quadrature is designed to reproduce the weighted angular moments of these vectors using only the selected angular directions. We choose a target angular space
\[
    \mathcal{Z}^n=\operatorname{span}(Z^n)\subset\mathbb{R}^{N_v}
\]
and require the sampled weights to satisfy the exactness condition
\begin{equation}\label{eq:sampling-exactness-condition}
    (Z^n)^\top w
    =
    (P_n^\top Z^n)^\top \widetilde{w}^n,
    \quad
    \text{i.e.}
    \quad
    \sum_{j=1}^{N_v} w_j z_j
    =
    \sum_{k=1}^{m_n}\widetilde{w}_k^n z_{i_k}
    \quad
    \forall\, z\in\mathcal Z^n.
\end{equation}
Consequently, if $h_i^n\in\mathcal Z^n$, then the sampled rule gives
$\widetilde{\mathcal J}_i^n=\mathcal J_i^n$.
The corresponding approximation estimate is given later in Lemma~\ref{lem:theory-transport-flux-error}.

For the default construction, we use
\[
    Z^n=Z^n_{(1)}
    =
    [\,\bm{1},\;Q\bm{1},\;QM^{-1}V^n\,]
    \in\mathbb R^{N_v\times(r+2)}.
\]
This choice is motivated by the zeroth-order angular structure of the backtracked flux derivative. Indeed, without the characteristic shift, the angular vector
\[
    \bigl(v_j\partial_x(f-\rho)(x_i,v_j,t_n)\bigr)_{j=1}^{N_v}
    \in
    \operatorname{span}(Q\bm{1},QM^{-1}V^n).
\]
The additional column $\bm{1}$ enforces preservation of the zeroth angular moment.

In the transition regime, the characteristic backtracking may introduce a non-negligible higher-order angular contribution. A Taylor expansion gives
    {\small\[
            v_j\partial_x(f-\rho)
            \left(x_i-\frac{v_j\Delta t}{\varepsilon},v_j,t_n\right)
            =
            v_j\partial_x(f-\rho)(x_i,v_j,t_n)
            -
            \frac{\Delta t}{\varepsilon}
            v_j^2\partial_{xx}(f-\rho)(x_i,v_j,t_n)
            +\cdots.
        \]}
The next-order term therefore has angular structure in
\[
    \operatorname{span}(Q^2\bm{1},Q^2M^{-1}V^n).
\]
This motivates the enlarged target
\[
    Z^n_{(2)}
    =
    [\,\bm{1},\;Q\bm{1},\;QM^{-1}V^n,\;Q^2\bm{1},\;Q^2M^{-1}V^n\,],
\]
which is used in the transition-regime test to better approximate the leading backtracking correction.

After the target matrix $Z^n$ is chosen, we use a QDEIM-type pivoting procedure to select the sampled directions.
Specifically, we apply column-pivoted QR to the transpose of $Z^n$:
\[
    (Z^n)^\top\Pi=\mathcal{Q}\mathcal{R},
    \qquad
    \Pi=[e_{\pi_1},\cdots,e_{\pi_{N_v}}].
\]
The sampled index set is given by the first $m_n$ pivot rows,
\[
    \mathcal{I}_n=\{\pi_1,\cdots,\pi_{m_n}\},
    \qquad
    P_n=[e_{\pi_1},\cdots,e_{\pi_{m_n}}].
\]
For the default target $Z^n_{(1)}$, we take $m_n=r+2$, while for the enlarged target $Z^n_{(2)}$, we take $m_n=2r+3$.

Once $Z^n$ and $P_n$ have been fixed, the effective quadrature weights are computed from the exactness equations \eqref{eq:sampling-exactness-condition}. In the square nonsingular case, this gives
\[
    \widetilde{w}^n
    =
    \left((P_n^\top Z^n)^\top\right)^{-1} (Z^n)^\top w.
\]
If this solve produces negative entries, we instead use the nonnegative least-squares solution
\[
    \min_{\widetilde{w}^n \ge 0}
    \norm*{(P_n^\top Z^n)^\top \widetilde{w}^n - (Z^n)^\top w}_2^2.
\]

\begin{remark}
    The angular sampling stage involves several implementation choices and optional extensions:
    \begin{enumerate}
        \item Before applying the QDEIM-type pivoting, one may compress $Z^n$ to an orthonormal basis of its numerical range. This removes nearly linearly dependent target directions and reduces possible scaling effects.
        \item The cost of the QDEIM selection and the effective-weight computation is not dominant in the overall low-rank update.
        \item The QDEIM-type selection can be replaced by greedy DEIM \cite{chaturantabut2010nonlinear}. Oversampling can also be incorporated by retaining more pivots than target columns and solving the exactness equations in a least-squares sense.
    \end{enumerate}
\end{remark}

\subsection{Multidimensional extensions}

The multidimensional extensions use the same matrix layout:
\[
    F^n=[f_1^n,\ldots,f_{N_\Omega}^n]\in\mathbb{R}^{N_s\times N_\Omega},
    \qquad
    Y^n=F^nM = X^nS^n(V^n)^\top.
\]
Here, rows correspond to spatial grid points, and columns correspond to discrete angular directions.
In the two- and three-dimensional tests, we use the $S_N$ angular discretization based on the Chebyshev--Legendre product quadrature \cite{lewis1983computational,larsen2009advances}.
For the present $S_N$ product rule, the order $N$ gives $N_\Omega=2N^2$ directions on the unit sphere, together with the associated weights.

The multidimensional scheme is constructed by applying the one-dimensional characteristic tracing separately in each spatial direction.
For each directional flux derivative, the corresponding upwind derivative of $f-\rho$ is first evaluated on the spatial grid and then interpolated to the characteristic foot point.
The macroscopic density solve and the BUG microscopic update are obtained by applying the one-dimensional formulas componentwise.
The target angular space is enlarged to include the coordinate moments required by the flux-derivative approximation. For example, in three dimensions, we use
\[
    Z^n_{(1)}
    =
    [\,\bm{1},\;Q_x\bm{1},\;Q_y\bm{1},\;Q_z\bm{1},
    \;Q_xM^{-1}V^n,\;Q_yM^{-1}V^n,\;Q_zM^{-1}V^n\,].
\]
The same QDEIM-based angular sample selection and effective-weight computation are then applied.

\begin{remark}
    In the multidimensional tests, we reorder the candidate angular directions before applying QDEIM by grouping directions that are related by coordinate reflections.
    This preprocessing helps prevent the sampled directions from concentrating near a particular angular region, without imposing any geometric symmetry constraint on the selected directions.
\end{remark}

\subsection{Per-Step Computational Complexity}\label{sec:complexity}

We compare the leading per-step costs of the full-rank semi-Lagrangian method (SL), the low-rank method with full macroscopic flux-derivative quadrature (SL-DLR(full)), and the sampled low-rank method (SL-DLR).
Let $N_s$ be the total number of spatial grid points, $N_\Omega$ the number of angular directions, and $r$ the rank.
In the sampled method, both the target-space dimension and the number of selected directions are $\mathcal{O}(r)$.
All implicit systems are solved by GMRES, with iteration counts $C_\rho$, $C_F$, $C_K$, $C_L$, and $C_S$ for the macroscopic solve,
microscopic column solves, and low-rank K-, L-, and S-steps.
The resulting leading costs are summarized in Table~\ref{tab:complexity}.
The table shows that SL-DLR(full) removes the full microscopic column solves, but it still assembles the backtracked macroscopic flux derivative over all angular directions.
Thus, applying a low-rank representation only to the microscopic component is not sufficient to reduce the leading complexity, since the semi-Lagrangian macroscopic flux derivative still requires a full angular sweep.
SL-DLR avoids this remaining full angular sweep by evaluating the macroscopic flux derivative only on the sampled angular set.
If each physical and angular coordinate is resolved by $N$ points, so $N_s=\mathcal{O}(N^{d_{\bm{x}}})$ and $N_\Omega=\mathcal{O}(N^{d_{\bm{\Omega}}})$,
then, for fixed rank and bounded GMRES iteration counts under mesh refinement, the leading costs scale as
\[
    \mathcal{C}_{\mathrm{SL}}
    =
    \mathcal{O}(N^{d_{\bm{x}}+d_{\bm{\Omega}}}),
    \quad
    \mathcal{C}_{\mathrm{SL\text{-}DLR(full)}}
    =
    \mathcal{O}(N^{d_{\bm{x}}+d_{\bm{\Omega}}}),
    \quad
    \mathcal{C}_{\mathrm{SL\text{-}DLR}}
    =
    \mathcal{O}(N^{\max\{d_{\bm{x}},d_{\bm{\Omega}}\}}).
\]

\begin{table}[htbp]
    \centering
    \caption{Leading per-step costs of the full-rank and low-rank schemes.}
    \label{tab:complexity}
    \small
    \setlength{\tabcolsep}{3pt}
    \begin{tabular}{llll}
        \toprule
        Stage & SL                                  & SL-DLR(full) & SL-DLR \\
        \midrule
        Sampling construction
              & --
              & --
              & $\mathcal{O}(N_\Omega r^2)$
        \\
        Macroscopic flux-derivative assembly
              & $\mathcal{O}(N_sN_\Omega)$
              & $\mathcal{O}(N_sN_\Omega r)$
              & $\mathcal{O}(N_s r^2)$
        \\
        Macroscopic solve
              & $\mathcal{O}(C_\rho N_s)$
              & $\mathcal{O}(C_\rho N_s)$
              & $\mathcal{O}(C_\rho N_s)$
        \\
        Microscopic column solves
              & $\mathcal{O}(C_F N_sN_\Omega)$
              & --
              & --
        \\
        K-step
              & --
              & $\mathcal{O}(C_K N_s r^2)$
              & $\mathcal{O}(C_K N_s r^2)$
        \\
        L-step
              & --
              & $\mathcal{O}(C_L N_\Omega r^2)$
              & $\mathcal{O}(C_L N_\Omega r^2)$
        \\
        S-step
              & --
              & $\mathcal{O}(C_S r^3)$
              & $\mathcal{O}(C_S r^3)$
        \\
        QR/projection
              & --
              & $\mathcal{O}(N_s r^2+N_\Omega r^2)$
              & $\mathcal{O}(N_s r^2+N_\Omega r^2)$
        \\
        \bottomrule
    \end{tabular}
\end{table}

\section{Theoretical Analysis}\label{sec:theoretical-analysis}

We analyze the proposed low-rank schemes in the one-dimensional periodic source-free setting, so that $\Phi\equiv0$.
The analysis has three components.
First, we prove the unconditional stability of the full-rank SL scheme for the constant-coefficient problem and show that this stability is inherited by SL-DLR(full).
Second, we estimate the perturbation caused by the sampled evaluation of the flux derivative in both the kinetic and diffusive regimes.
Third, we prove the AP property of SL-DLR by combining the AP limit of the full-rank SL scheme, the equilibrium structure of the BUG update, and the sampling estimate.

For vectors and matrices, we use the discrete inner products and norms
\[
    \begin{gathered}
        \inner*{u,v}_{\Delta x}
        =
        \Delta x\sum_i u_i v_i,
        \qquad
        \norm*{u}_{\Delta x}^2
        =
        \inner*{u,u}_{\Delta x},
        \\
        \inner*{Y,Z}_{F,\Delta x}
        =
        \Delta x\sum_i\sum_j Y_{ij}Z_{ij},
        \qquad
        \norm*{Y}_{F,\Delta x}^2
        =
        \inner*{Y,Y}_{F,\Delta x}.
    \end{gathered}
\]

\subsection{Energy stability}

The energy stability analysis is restricted to the constant-coefficient case.
In this setting, $\sigma^s$, $\sigma^a$, $\alpha_1$, $\alpha_2$, and $\beta$ are scalars, and $L_\beta=D_x^+(\beta D_x^-)$.

\begin{proposition} \label{prop:theory-fullrank-energy-stability}
    In the one-dimensional periodic source-free constant-coefficient setting, let $F^{n+1}$ be obtained from $F^n$ by one step of the full-rank FD-SL scheme with time step $\Delta t > 0$. Then
    \[
        \norm*{F^{n+1}M}_{F,\Delta x}^2 \le \norm*{F^n M}_{F,\Delta x}^2.
    \]
\end{proposition}

For SL-DLR(full), the macroscopic flux derivative is kept unchanged, while the full-rank update is replaced by the BUG integrator.

\begin{proposition} \label{prop:theory-lowrank-full-energy-stability}
    Under the setting of Proposition~\ref{prop:theory-fullrank-energy-stability},
    let $Y_{\mathrm{SL\text{-}DLR(full)}}^{n+1}$ be obtained from $Y^n=X^nS^n(V^n)^\top$ by one step of SL-DLR(full) with time step $\Delta t > 0$. Then
    \[
        \norm*{
            Y_{\mathrm{SL\text{-}DLR(full)}}^{n+1}
        }_{F,\Delta x}^2
        \le
        \norm*{Y^n}_{F,\Delta x}^2.
    \]
\end{proposition}

The proofs of Propositions~\ref{prop:theory-fullrank-energy-stability} and
\ref{prop:theory-lowrank-full-energy-stability} are given in
Appendix~\ref{app:energy-stability}.

\subsection{Sampling error}

We first compare the full-quadrature flux derivative $\mathcal{J}^n$ with the sampled flux derivative $\widetilde{\mathcal{J}}^n$,
and then propagate
$\eta_n := \norm*{\widetilde{\mathcal{J}}^n-\mathcal{J}^n}_{\Delta x}$
through the macroscopic density solve.

\subsubsection{Kinetic regime}

In the kinetic regime $\varepsilon=\mathcal{O}(1)$, the sampling error is mainly caused by the velocity-dependent characteristic backtracking in the flux derivative.

\begin{lemma}\label{lem:theory-transport-flux-error}
    Assume that the sampled weights satisfy the exactness condition
    \eqref{eq:sampling-exactness-condition}, and that
    $\norm*{\widetilde w^n}_1\le C_{\widetilde w}$,
    where $C_{\widetilde w}$ is independent of
    $\varepsilon$, $\Delta x$, and $\Delta t$.
    Then
    \begin{equation}\label{eq:theory-transport-eta}
        \eta_n
        \le
        (1+C_{\widetilde w})
        \left(
        \Delta x\sum_{i=1}^{N_x}
        \inf_{\xi \in\mathcal{Z}^n}\norm*{h_i^n-\xi }_\infty^2
        \right)^{1/2}.
    \end{equation}
\end{lemma}

\begin{proof}
    For any $\xi \in\mathcal{Z}^n$, the exactness condition
    \eqref{eq:sampling-exactness-condition} gives $w^\top \xi =  (\widetilde w^n)^\top P_n^\top \xi$.
    Therefore,
    \[
        \left|\widetilde{\mathcal{J}}_i^n-\mathcal{J}_i^n\right|
        =
        \left|
        (\widetilde w^n)^\top P_n^\top(h_i^n-\xi)
        -
        w^\top(h_i^n-\xi)
        \right|
        \le
        (1+C_{\widetilde w})\norm*{h_i^n-\xi}_\infty.
    \]
    Taking the infimum over $\xi \in\mathcal{Z}^n$ and then summing over $i$ in
    $\norm*{\,\cdot\,}_{\Delta x}$ gives \eqref{eq:theory-transport-eta}.
\end{proof}

Estimate \eqref{eq:theory-transport-eta} reduces the sampling error to the approximation of $h_i^n$ in the target space. We identify its leading angular structure by a Taylor expansion.
For an integer $p\ge0$, define the backtracking regularity measure
\[
    \Theta_p^n
    :=
    \left(
    \Delta x\sum_{i=1}^{N_x}
    \left[
        \max_j
        \left|
        |v_j|^{p+1}
        \partial_x^{p+1}(f-\rho)(x_i,v_j,t_n)
        \right|
        \right]^2
    \right)^{1/2}.
\]
For smooth data, the FD-SL derivative reconstruction and interpolation are first-order accurate.
Thus, up to an $\mathcal{O}(\Delta x)$ contribution, $h_i^n$ has the same leading angular structure as the smooth backtracked vector
\[
    h_{i,c}^n(v_j)
    :=
    v_j\partial_x(f-\rho)
    \left(x_i-\frac{v_j\Delta t}{\varepsilon},v_j,t_n\right).
\]
A Taylor expansion at $x_i$ gives
\[
    h_{i,c}^n(v_j)
    =
    v_j\partial_x(f-\rho)(x_i,v_j,t_n)
    -
    \frac{\Delta t}{\varepsilon}
    v_j^2\partial_{xx}(f-\rho)(x_i,v_j,t_n)
    +
    \mathcal{O}\!\left(\frac{\Delta t^2}{\varepsilon^2}\right).
\]
The zeroth-order angular vector
\[
    \bar{h}_{i,c}^n
    :=
    \bigl(v_j\partial_x(f-\rho)(x_i,v_j,t_n)\bigr)_{j=1}^{N_v}
    \in
    \spanop(Q\bm{1},QM^{-1}V^n)
    \subset
    \mathcal{Z}_{(1)}^n.
\]
Thus the zeroth-order term is contained in the default target space.
The remaining contribution comes from the first backtracking correction and from the first-order FD-SL reconstruction error.
Applying Lemma~\ref{lem:theory-transport-flux-error} with
$\mathcal{Z}^n=\mathcal{Z}_{(1)}^n$ yields
\begin{equation}\label{eq:theory-transport-eta-discrete}
    \eta_n
    \le
    (1+C_{\widetilde w})
    \left(
    \frac{\Delta t}{\varepsilon}
    \Theta_1^n
    +
    C\left(\frac{\Delta t^2}{\varepsilon^2}+\Delta x\right)
    \right),
\end{equation}
where $C$ depends on the relevant higher spatial smoothness bounds, but not on
$\Delta t$, $\Delta x$, or $\varepsilon$.
If the enlarged target space $\mathcal{Z}^n=\mathcal{Z}_{(2)}^n$ is used, the first backtracking correction is also represented in the target space, and the same estimate gives
\[
    \eta_n
    \le
    (1+C_{\widetilde w})
    \left(
    \frac{\Delta t^2}{2\varepsilon^2}
    \Theta_2^n
    +
    C\left(\frac{\Delta t^3}{\varepsilon^3}+\Delta x\right)
    \right).
\]

\subsubsection{Diffusive limit}

In the diffusive limit $\varepsilon\to0$, the sampling error in the backtracked semi-Lagrangian flux derivative is not necessarily small.
The key point is that this perturbation enters the macroscopic equation through the prefactor $\diag(\alpha_1)/\varepsilon$.
Since $\norm{\alpha_1}_{\infty} / \varepsilon \to 0$ as $\varepsilon\to0$,
the effect of any uniformly bounded sampling error on the macroscopic density is exponentially suppressed in the diffusive limit.

\begin{lemma}\label{lem:theory-small-eps-sampling-bounded}
    Assume that the discrete nonequilibrium part and the effective quadrature weights satisfy
    \[
        \max_{1\le j\le N_v} \norm*{f_j^n-\rho^n}_\infty \le C_g,
        \qquad
        \norm*{\widetilde{w}^n}_1 \le C_{\widetilde{w}},
    \]
    where $C_g$ and $C_{\widetilde{w}}$ are independent of $\varepsilon$, $\Delta t$, and $\Delta x$.
    Then the sampling error satisfies
    \[
        \eta_n
        =
        \norm*{\widetilde{\mathcal{J}}^n-\mathcal{J}^n}_{\Delta x}
        \le
        \frac{C_\eta}{\Delta x},
    \]
    where $C_\eta$ is independent of $\varepsilon$, $\Delta t$, and $\Delta x$.
\end{lemma}

\begin{proof}
    The backtracked flux derivative is obtained by applying one-sided difference quotients to $f_j^n-\rho^n$ and then evaluating the result by linear interpolation.
    Thus,
    \[
        \norm*{
        \mathcal{B}_j(D_x^{\mathrm{up},(j)}(f_j^n-\rho^n))
        }_\infty
        \le
        \frac{2C_g}{\Delta x}.
    \]
    Using $|v_j|\le1$, $\norm*{w}_1 = \sum_j w_j = 1$ and $\norm*{\widetilde{w}^n}_1 \le C_{\widetilde{w}}$, we obtain
    \[
        \norm*{\mathcal{J}^n}_\infty
        \le
        \frac{2C_g}{\Delta x}
        \sum_{j=1}^{N_v}w_j|v_j|
        \le
        \frac{2C_g}{\Delta x},
        \quad
        \norm*{\widetilde{\mathcal{J}}^n}_\infty
        \le
        \frac{2C_g}{\Delta x}
        \sum_{k=1}^{m_n}|\widetilde w_k^n|\,|v_{i_k}|
        \le
        \frac{2C_g C_{\widetilde w}}{\Delta x}.
    \]
    Therefore,
    \[
        \norm*{\widetilde{\mathcal{J}}^n-\mathcal{J}^n}_{\Delta x}
        \le
        \sqrt{x_R-x_L}\,\norm*{\widetilde{\mathcal{J}}^n-\mathcal{J}^n}_\infty
        \le
        \frac{2\sqrt{x_R-x_L}\,C_g(1+C_{\widetilde w})}{\Delta x}.
    \]
    This gives the desired bound.
\end{proof}

\subsubsection{Effect on the macroscopic density solve}

The sampling error in the flux derivative enters the implicit macroscopic density solve only through the right-hand side.
Let $\rho_{\mathrm{SL\text{-}DLR(full)}}^{n+1,*}$ and $\rho_{\mathrm{SL\text{-}DLR}}^{n+1,*}$ denote the provisional densities obtained from \eqref{eq:macroscopic-linear-system} and \eqref{eq:sampled-macroscopic-linear-system},
respectively.
Then their difference satisfies
\begin{equation}\label{eq:theory-delta-rho-perturbation}
    A_\rho
    (\rho_{\mathrm{SL\text{-}DLR}}^{n+1,*}
    -\rho_{\mathrm{SL\text{-}DLR(full)}}^{n+1,*})
    =
    -\frac{\Delta t}{\varepsilon}\diag(\alpha_1)
    (\widetilde{\mathcal{J}}^n-\mathcal{J}^n),
\end{equation}
where $A_\rho=I-\Delta t L_\beta+\Delta t\diag(\sigma^a)$.

\begin{lemma}\label{lem:theory-bounded-inverse-rescaled-op}
    The operator $A_\rho$ is invertible and satisfies $\norm*{A_\rho^{-1}}_{\Delta x\to\Delta x}\le 1$.
\end{lemma}

\begin{proof}
    By periodicity and $\beta_{i+\frac12}\ge0$, $\sigma_i^a\ge0$, we have
    \[
        \inner*{A_\rho u,u}_{\Delta x}
        =
        \norm*{u}_{\Delta x}^2
        +
        \Delta t\inner*{-L_\beta u,u}_{\Delta x}
        +
        \Delta t\inner*{\diag(\sigma^a)u,u}_{\Delta x}
        \ge
        \norm*{u}_{\Delta x}^2.
    \]
    Hence, by Cauchy--Schwarz, $\norm*{A_\rho u}_{\Delta x}\ge\norm*{u}_{\Delta x}$ for all $u$.
    This implies that $A_\rho$ is invertible and $\norm*{A_\rho^{-1}}_{\Delta x\to\Delta x}\le 1$.
\end{proof}

Combining the perturbation equation for the macroscopic solve with the preceding flux-error estimates gives the following one-step bound.

\begin{proposition} \label{prop:theory-delta-rho-bound}
    The one-step perturbation caused by replacing $\mathcal{J}^n$ with $\widetilde{\mathcal{J}}^n$ satisfies
    \begin{equation}\label{eq:theory-delta-rho-basic-bound}
        \norm*{\rho_{\mathrm{SL\text{-}DLR}}^{n+1,*}
            -\rho_{\mathrm{SL\text{-}DLR(full)}}^{n+1,*}}_{\Delta x}
        \le
        \frac{\Delta t\norm*{\alpha_1}_\infty}{\varepsilon}\eta_n.
    \end{equation}
    Moreover:
    \begin{enumerate}
        \item In the kinetic regime, if the assumptions leading to
              \eqref{eq:theory-transport-eta-discrete} hold, then
              \begin{equation}\label{eq:theory-delta-rho-kinetic}
                  \begin{aligned}
                       & \norm*{\rho_{\mathrm{SL\text{-}DLR}}^{n+1,*}
                          -\rho_{\mathrm{SL\text{-}DLR(full)}}^{n+1,*}}_{\Delta x}
                      \\
                       & \le{}
                      \frac{\Delta t\norm*{\alpha_1}_\infty}{\varepsilon}
                      (1+C_{\widetilde w})
                      \left(
                      \frac{\Delta t}{\varepsilon}
                      \Theta_{1}^n
                      +
                      C\left(\frac{\Delta t^2}{\varepsilon^2}+\Delta x\right)
                      \right).
                  \end{aligned}
              \end{equation}
              In particular, for $\varepsilon=\mathcal{O}(1)$ and sufficiently smooth solutions, the leading macroscopic perturbation is $\mathcal{O}(\Delta t^2+\Delta t\,\Delta x)$.

        \item In the diffusive limit, if the assumptions of Lemma~\ref{lem:theory-small-eps-sampling-bounded} hold, then, for fixed $\Delta t$ and mesh,
              \begin{equation}\label{eq:theory-delta-rho-diffusive}
                  \norm*{\rho_{\mathrm{SL\text{-}DLR}}^{n+1,*} -\rho_{\mathrm{SL\text{-}DLR(full)}}^{n+1,*}}_{\Delta x} \le \frac{C_\eta \Delta t}{\Delta x} \frac{\norm*{\alpha_1}_\infty}{\varepsilon} \to 0, \qquad \varepsilon\to0.
              \end{equation}
              Thus, the sampled-flux perturbation does not affect the diffusive limit.
    \end{enumerate}
\end{proposition}

\begin{proof}
    Applying $A_\rho^{-1}$ to \eqref{eq:theory-delta-rho-perturbation} and using Lemma~\ref{lem:theory-bounded-inverse-rescaled-op} give \eqref{eq:theory-delta-rho-basic-bound}.
    The kinetic-regime estimate \eqref{eq:theory-delta-rho-kinetic} follows from \eqref{eq:theory-delta-rho-basic-bound} and \eqref{eq:theory-transport-eta-discrete}.
    The diffusive-limit estimate \eqref{eq:theory-delta-rho-diffusive} follows from \eqref{eq:theory-delta-rho-basic-bound}, Lemma~\ref{lem:theory-small-eps-sampling-bounded}, and $\norm*{\alpha_1}_\infty/\varepsilon\to0$.
\end{proof}

\subsection{Asymptotic-preserving property}

The full-rank FD-SL scheme is AP \cite{zhang2023asymptotic}.
We show that the proposed low-rank schemes preserve the same diffusive limit.

\begin{proposition}\label{prop:theory-lowrank-ap}
    Consider one time step of the low-rank FD-SL schemes in the source-free case.
    Assume that $\Delta t$ and the mesh are fixed, and that the recovered density satisfies $\rho^n\ge0$ and $\norm*{\rho^n}_{\Delta x} > 0$.
    Assume further that the sampled flux derivative satisfies the boundedness condition in Lemma~\ref{lem:theory-small-eps-sampling-bounded}.
    Then both SL-DLR(full) and SL-DLR are AP.
    More precisely,
    if $\rho^{n+1}_0$ denotes the limiting diffusion update, then, as $\varepsilon\to0$,
    \[
        Y_{\mathrm{SL\text{-}DLR(full)}}^{n+1} = \rho^{n+1}_0 (M\bm{1})^\top + o(1).
    \]
    Moreover, SL-DLR has the same limiting diffusion equation as SL-DLR(full).
\end{proposition}

\begin{proof}
    Since $\rho^n=X^nS^n(V^n)^\top M\bm{1}$ and $\norm*{\rho^n}_{\Delta x} > 0$, we have
    \begin{equation}\label{eq:theory-rho-nonzero-1}
        (M\bm{1})^\top V^n \ne 0,
    \end{equation}
    and $\rho^n\in\spanop(X^n)$.
    In the diffusive limit, $\rho^{n+1,*}=\rho_0^{n+1}+o(1)$.
    Hence, for sufficiently small $\varepsilon$, $\norm*{\rho^{n+1,*}}_{\Delta x} > 0$ and
    \begin{equation}\label{eq:theory-rho-nonzero-2}
        (X^n)^\top\diag(\sigma^s)\rho^{n+1,*}\ne0.
    \end{equation}

    The leading $\varepsilon^{-2}$ K-step balance and L-step balance give
    \begin{align*}
        K^{n+1}
         & =
        \rho^{n+1,*} (M\bm{1})^\top V^n + o(1),
        \\
        L^{n+1}(X^n)^\top\diag(\sigma^s)X^n
         & =
        M\bm{1}\,(\diag(\sigma^s)\rho^{n+1,*})^\top X^n + o(1).
    \end{align*}
    Together with \eqref{eq:theory-rho-nonzero-1} and \eqref{eq:theory-rho-nonzero-2}, these balances imply
    \[
        \rho^{n+1,*}\in\spanop(X^{n+1}) + o(1),
        \quad
        M\bm{1}\in\spanop(V^{n+1}) + o(1).
    \]
    The leading S-step balance is
    \[
        (X^{n+1})^\top\diag(\sigma^s)X^{n+1}S^{n+1}
        =
        (X^{n+1})^\top\diag(\sigma^s)\rho^{n+1,*} (M\bm{1})^\top V^{n+1} + o(1).
    \]
    Let $A_X^{n+1}=(X^{n+1})^\top\diag(\sigma^s)X^{n+1}$. Then
    \begin{align*}
        Y_{\mathrm{SL\text{-}DLR(full)}}^{n+1}
        ={} &
        X^{n+1}S^{n+1}(V^{n+1})^\top
        \\
        ={} &
        X^{n+1}(A_X^{n+1})^{-1} (X^{n+1})^\top\diag(\sigma^s)\rho^{n+1,*} (M\bm{1})^\top V^{n+1}(V^{n+1})^\top + o(1)
        \\
        ={} &
        \rho^{n+1,*} (M\bm{1})^\top + o(1)
        \\
        ={} &
        \rho^{n+1}_0 (M\bm{1})^\top + o(1).
    \end{align*}
    Thus, SL-DLR(full) inherits the limiting diffusion equation.
    Finally, the transition from SL-DLR(full) to SL-DLR follows from Proposition~\ref{prop:theory-delta-rho-bound}.
\end{proof}

\section{Numerical results}\label{sec:numerical-results}

We now examine the accuracy, asymptotic behavior, and computational cost of the proposed sampled low-rank method.
We compare the full-rank semi-Lagrangian scheme (SL), the sampled low-rank scheme (SL-DLR), and the full-quadrature low-rank reference, denoted by SL-DLR(full). The latter uses the same low-rank evolution as SL-DLR but evaluates the macroscopic flux derivative with all angular directions. Therefore, SL-DLR(full) isolates the fixed-rank approximation error, whereas SL-DLR additionally contains the flux-derivative error introduced by angular sampling. Unless otherwise stated, all experiments use periodic boundary conditions, set $\sigma^a=\Phi=0$, and prescribe a fixed rank throughout each run.
In the multidimensional tests, the angular discretization is the Chebyshev--Legendre product $S_m$ rule with $N_\Omega=2m^2$ directions.
All implicit systems are solved by GMRES, implemented using the MATLAB \texttt{gmres} function. Unless otherwise stated, Jacobi preconditioning and a relative tolerance of $10^{-9}$ are used for all solves.

\subsection{Gaussian initial data in one dimension}\label{ex:1d-gaussian}

We investigate the radiation transport equation in slab geometry in the spatial domain $D_x = [-1.5, 1.5]$
using the Gaussian initial condition
\[
    f(0,x,v)
    =
    \frac{1}{\sqrt{2\pi}\Sigma}
    \exp\left(-\frac{x^2}{2\Sigma^2}\right),
    \quad
    \Sigma^2=9\times10^{-4}.
\]
We set $\sigma^s = 1$ and fix $(N_{v}, N_x) = (200, 500)$ for all regimes.
Figure~\ref{fig:1d-gaussian-density} displays the density profiles for the kinetic ($\varepsilon=1, T=1, \Delta t = 2\Delta x, r=50$), transition ($\varepsilon=10^{-2}, T=0.2, \Delta t = \Delta x, r=8$), and diffusive ($\varepsilon=10^{-6}, T=0.2, \Delta t = \Delta x, r=3$) regimes.
The diffusion-limit reference is also shown in the diffusive regime.
In all three regimes, both SL-DLR(full) and SL-DLR closely match the SL density.
In the kinetic regime $(\varepsilon=1)$, the numerical profiles follow the trend of the analytical solution \cite{ganapol2008analytical}, with the remaining discrepancy mainly reflecting the low-order discretization.
In the diffusive regime $(\varepsilon=10^{-6})$, the low-rank results also agree with the limiting diffusion solution, which is consistent with the AP behavior of the full-rank SL scheme.

\begin{figure}[htbp!]
    \centering
    \begin{subfigure}[b]{0.32\textwidth}
        \centering
        \includegraphics[width=\textwidth]{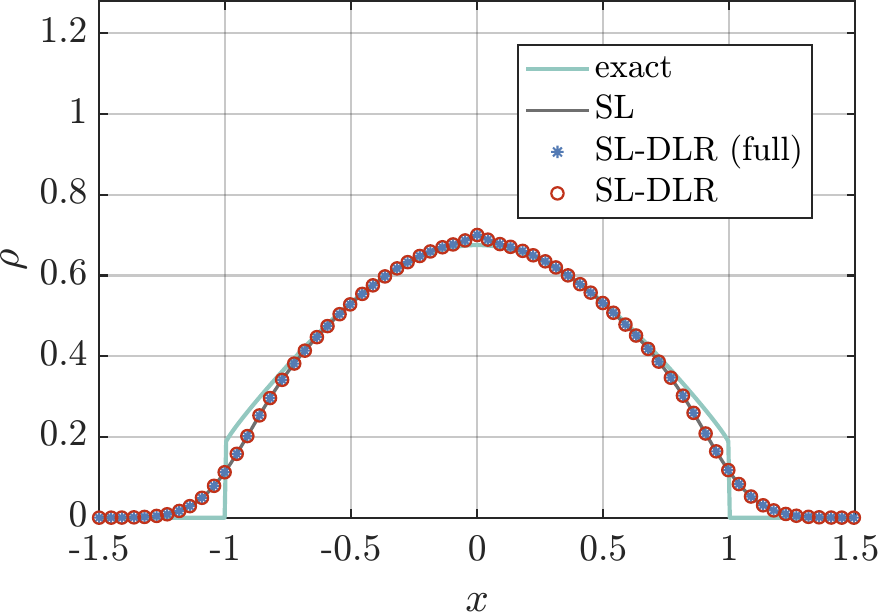}
    \end{subfigure}
    \begin{subfigure}[b]{0.32\textwidth}
        \centering
        \includegraphics[width=\textwidth]{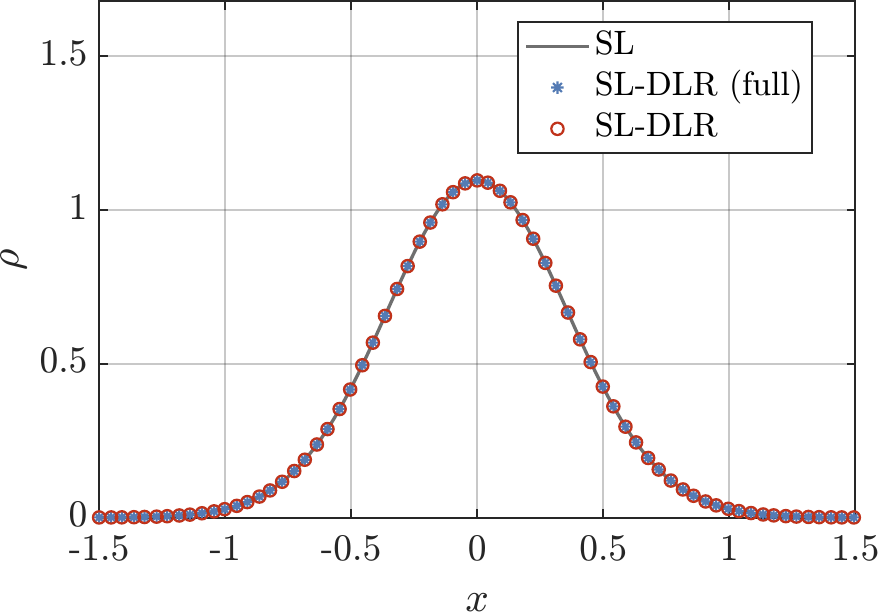}
    \end{subfigure}
    \begin{subfigure}[b]{0.32\textwidth}
        \centering
        \includegraphics[width=\textwidth]{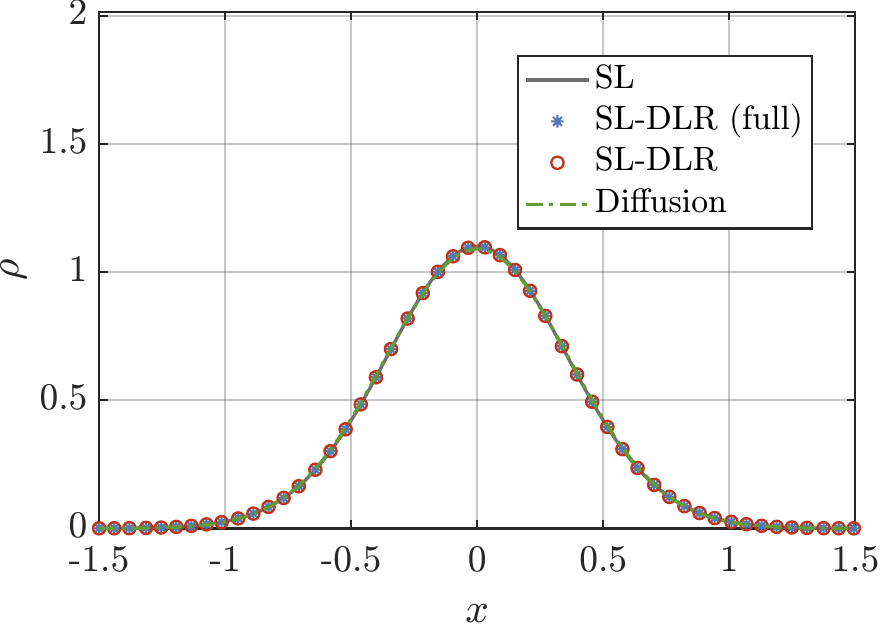}
    \end{subfigure}
    \caption{Example~\ref{ex:1d-gaussian}.
    Density profiles in the kinetic ($\varepsilon=1$), transition ($\varepsilon=10^{-2}$), and diffusive regimes ($\varepsilon=10^{-6}$), shown from left to right. }
    \label{fig:1d-gaussian-density}
\end{figure}

\subsection{One-dimensional non-equilibrium data and flux-derivative sampling error}\label{ex:1d-nonequilibrium}

We consider a non-equilibrium initial condition in the transition regime, where the analysis in Section~\ref{sec:theoretical-analysis} suggests that the sampled macroscopic flux derivative is more sensitive to the angular target space.
This example is used to examine the effect of flux-derivative sampling and target-space enrichment on the density.
We take $\varepsilon=0.1$ and use the initial condition
\[
    f(0,x,v)
    =
    \frac{1}{2\pi\Sigma^2}
    \exp\left(-\frac{x^2}{2\Sigma^2}\right)
    \exp\left(-\frac{(v-1)^2}{2\Sigma^2}\right),
    \quad
    \Sigma^2=9\times10^{-4}.
\]
We set $\sigma^s = 1$ and choose $N_x=200$, $N_v=100$, and $r=6$.
The final time is $T=0.2$, and the time step is $\Delta t=\Delta x$.
We compare SL, SL-DLR(full), SL-DLR with the default target basis $Z^n=Z^n_{(1)}$, and SL-DLR with the enlarged target basis $Z^n=Z^n_{(2)}$ suggested by the backtracking Taylor expansion.
As explained in Section~\ref{sec:sampling}, $Z^n_{(2)}$ includes the leading $Q^2$-weighted backtracking term, thereby reducing the leading sampling contribution predicted by the Taylor expansion.
Figure~\ref{fig:1d-nonequilibrium-zn} shows the density comparison. SL-DLR(full) remains close to SL, indicating that the fixed-rank approximation error is not dominant in this test.
By contrast, the density obtained by the default sampled SL-DLR differs visibly from that of SL-DLR(full), consistent with the sampling-error estimate.
Replacing $Z^n_{(1)}$ by $Z^n_{(2)}$ reduces this difference without reverting to the full angular quadrature.

\begin{figure}[htbp!]
    \centering
    \begin{subfigure}[c]{0.5\textwidth}
        \centering
        \includegraphics[width=\textwidth]{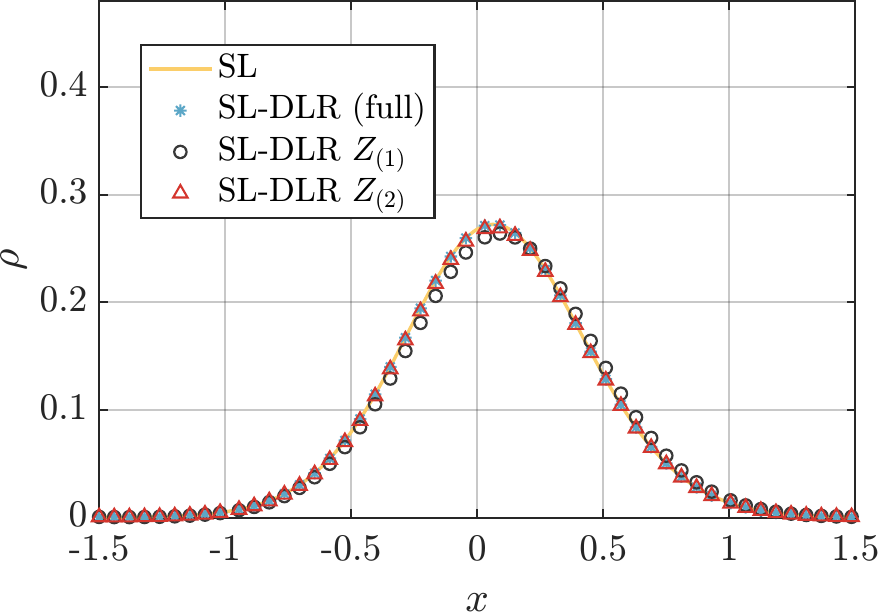}
    \end{subfigure}
    \caption{Example~\ref{ex:1d-nonequilibrium}.
    Density comparison for SL, SL-DLR(full), SL-DLR with default target basis $Z^n_{(1)}$, and SL-DLR with the enlarged target basis $Z^n_{(2)}$.}
    \label{fig:1d-nonequilibrium-zn}
\end{figure}

\subsection{Accuracy and computational cost in two dimensions}\label{ex:2d-accuracy-cost}

We verify the accuracy and computational cost of SL-DLR and SL-DLR(full) in two dimensions using a low-rank manufactured solution of the form
\begin{equation}\label{eq:2d-mms-f}
    f(t,x,y,\theta,\mu)
    =
    2+e^{-t}\sin(2\pi x)\sin(2\pi y)
    \bigl(1+\varepsilon\sin\theta\sqrt{1-\mu^2}\bigr),
\end{equation}
where $(x,y) \in [0,1]^2$.
We set $\sigma^s=1$ in this test.
The source term is obtained by substituting \eqref{eq:2d-mms-f} into the model equation.
Since this source is velocity-dependent, we use the velocity-dependent source treatment described in Remark~\ref{rem:source-velocity-dependent}.
The test is performed for $\varepsilon=1$, $10^{-2}$, and $10^{-6}$.
The spatial mesh is refined as $N_x=N_y=N \in \{16,32,64,128\}$, with the $S_{N/4}$ angular rule and rank $r=8$.
We set $\Delta t=2\Delta x$ and compute to $T=1.0$.
Figure~\ref{fig:2d-order} shows the $L^1$ error in $\rho$ and the average wall time per step.
For all three values of $\varepsilon$, SL and SL-DLR give comparable errors and convergence rates.
Thus, the fixed-rank approximation and QDEIM sampling do not degrade the observed accuracy in this test.
The timing slopes agree with the complexity discussion in Section~\ref{sec:complexity}: along this refinement path, SL and SL-DLR(full) retain the $\mathcal{O}(N^4)$ cost, whereas SL-DLR exhibits the reduced $\mathcal{O}(N^2)$ scaling.

\begin{figure}[htbp!]
    \centering
    \begin{subfigure}[c]{0.32\textwidth}
        \centering
        \includegraphics[width=\textwidth]{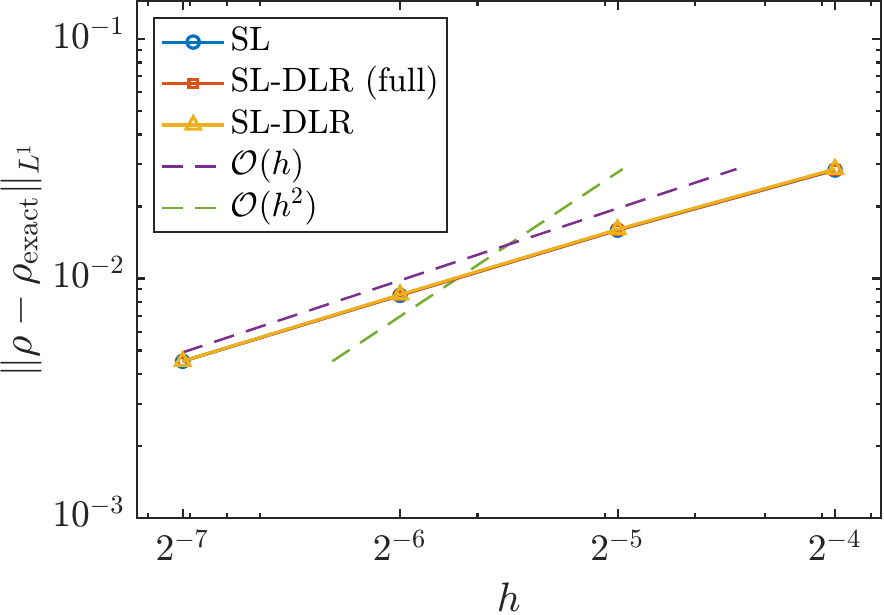}
    \end{subfigure}
    \begin{subfigure}[c]{0.32\textwidth}
        \centering
        \includegraphics[width=\textwidth]{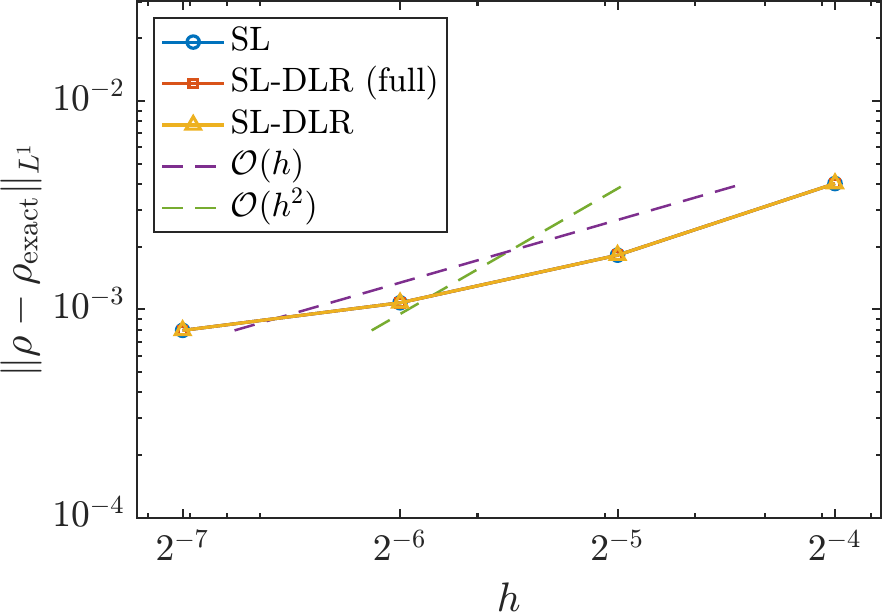}
    \end{subfigure}
    \begin{subfigure}[c]{0.32\textwidth}
        \centering
        \includegraphics[width=\textwidth]{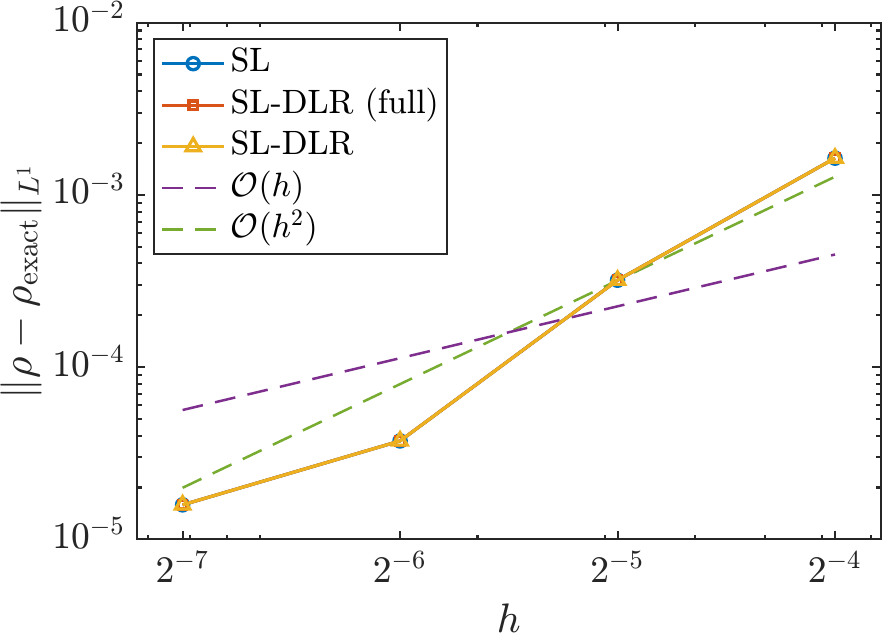}
    \end{subfigure}

    \begin{subfigure}[c]{0.32\textwidth}
        \centering
        \includegraphics[width=\textwidth]{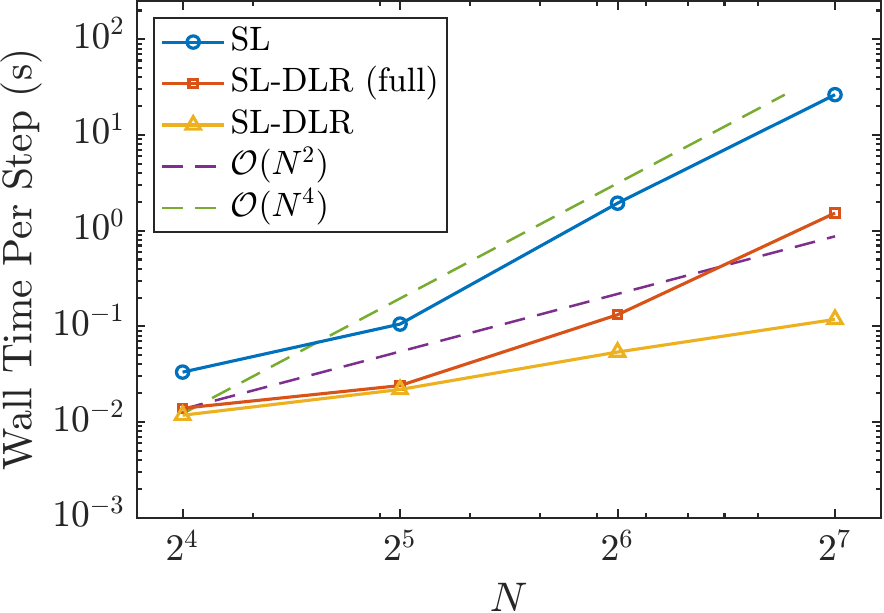}
    \end{subfigure}
    \begin{subfigure}[c]{0.32\textwidth}
        \centering
        \includegraphics[width=\textwidth]{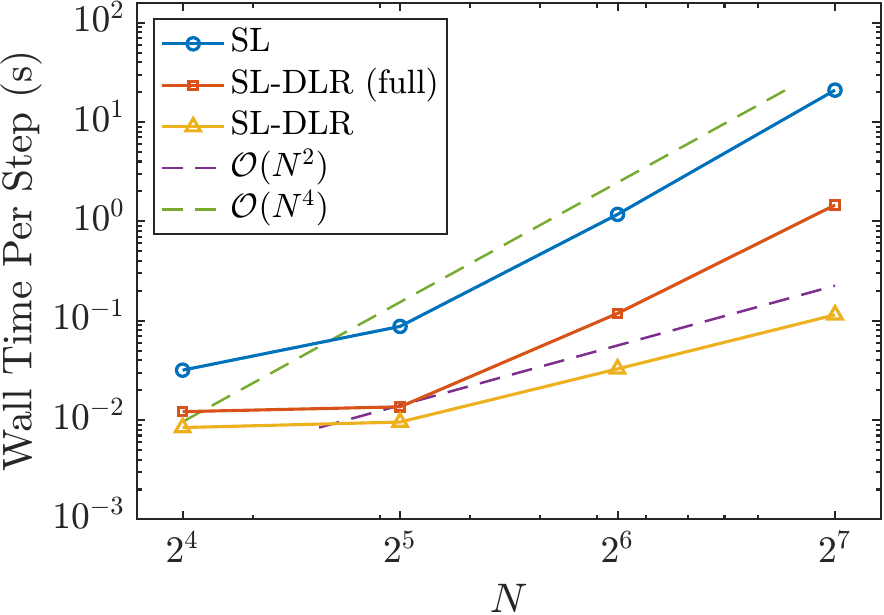}
    \end{subfigure}
    \begin{subfigure}[c]{0.32\textwidth}
        \centering
        \includegraphics[width=\textwidth]{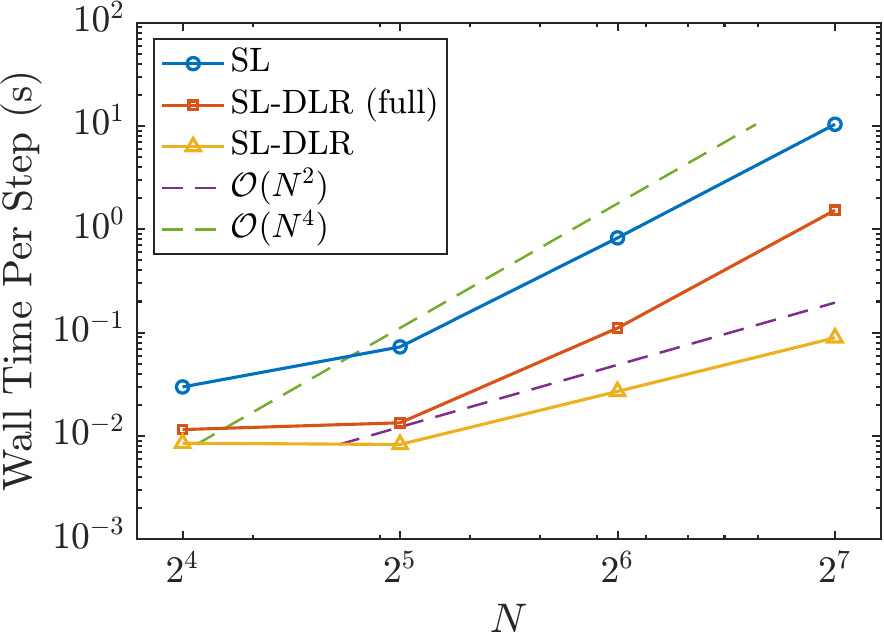}
    \end{subfigure}

    \caption{Example~\ref{ex:2d-accuracy-cost}.
        Top row: $L^1$ density error.
        Bottom row: average wall time per step.
        From left to right, $\varepsilon=1$, $10^{-2}$, and $10^{-6}$.}
    \label{fig:2d-order}
\end{figure}

\subsection{Gaussian initial data in two dimensions}\label{ex:2d-gaussian}

We use a smooth Gaussian initial condition on $[-1,1]^2$:
\[
    f(0,x,y,\bm{\Omega})
    =
    \frac{1}{4\pi\Sigma^2}
    \exp\left(-\frac{x^2+y^2}{4\Sigma^2}\right),
    \quad
    \Sigma^2=10^{-2}.
\]
We set $\sigma^s=1$, take $\varepsilon=10^{-6}$, and compute to $T=0.1$.
The discretization uses $N_x=N_y=64$, the $S_{16}$ angular rule with $N_\Omega=512$, and rank $r=4$.
The time step is $\Delta t=0.01\approx 10\Delta x^2$.
Figure~\ref{fig:2d-gaussian} shows that the low-rank methods are consistent with both the full-rank SL reference and the diffusion-limit solution.
It also shows the angular directions used in SL-DLR at the final time step.

\begin{figure}[htbp!]
    \centering
    \begin{subfigure}[b]{0.32\textwidth}
        \centering
        \includegraphics[width=\textwidth]{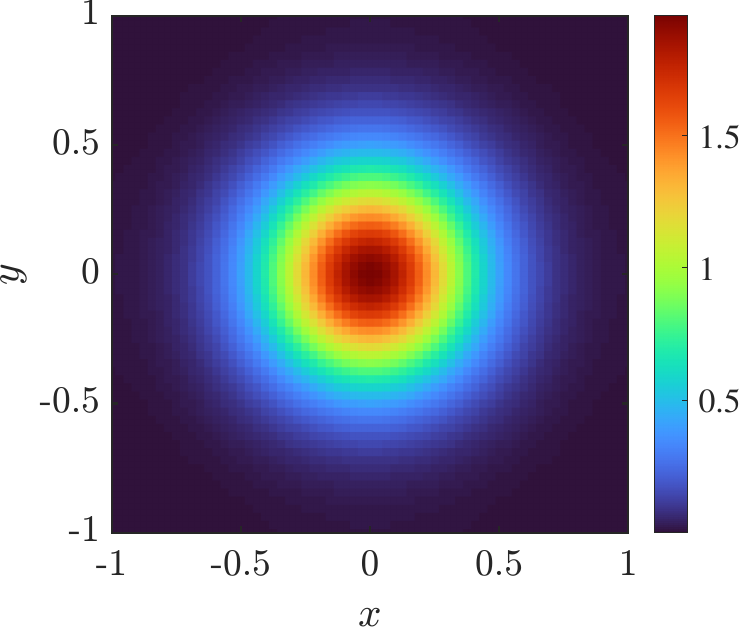}
    \end{subfigure}
    \begin{subfigure}[b]{0.34\textwidth}
        \centering
        \includegraphics[width=\textwidth]{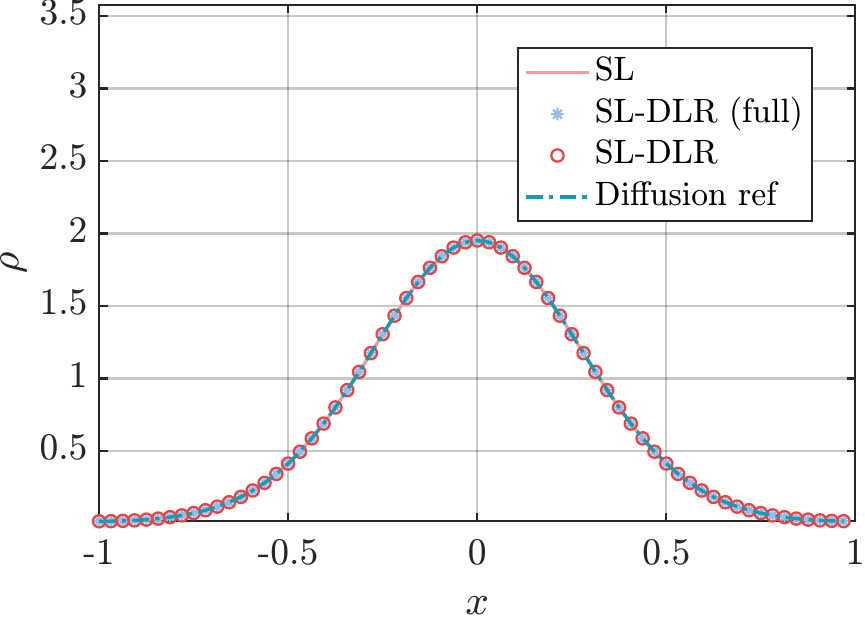}
    \end{subfigure}
    \begin{subfigure}[b]{0.28\textwidth}
        \centering
        \includegraphics[width=\textwidth]{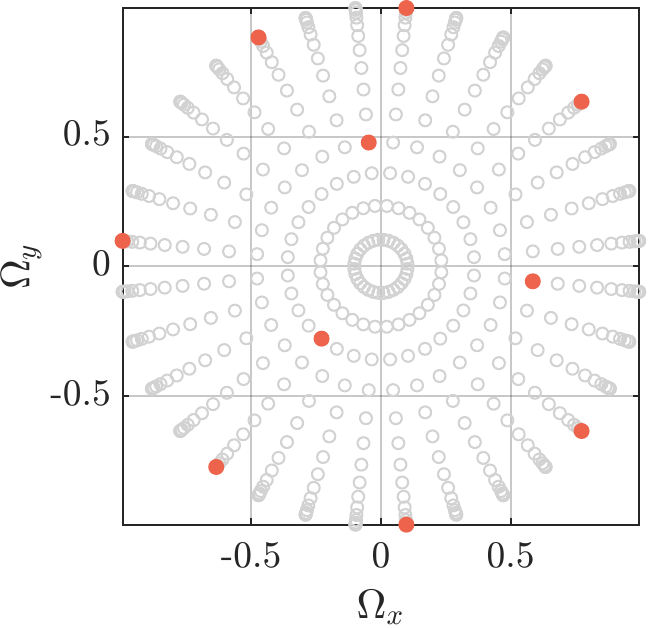}
    \end{subfigure}
    \caption{Example~\ref{ex:2d-gaussian}.
        SL-DLR density (left), density profile along $y=0$ (middle), and sampled directions (right; filled circles: selected, open circles: unselected).
    }
    \label{fig:2d-gaussian}
\end{figure}

\subsection{Two-dimensional Gaussian data with varying scattering}\label{ex:2d-variable-sigma}

We next keep the Gaussian initial condition used in Example~\ref{ex:2d-gaussian} but replace the constant scattering coefficient with a varying radial coefficient:
\[
    \sigma^s(x,y)
    =
    \begin{cases}
        0.999\,c^4(c+\sqrt{2})^2(c-\sqrt{2})^2+0.001,
           & c=\sqrt{x^2+y^2}<1, \\
        1, & \text{otherwise}.
    \end{cases}
\]
We set $\varepsilon=10^{-2}$, for which $\sigma^s/\varepsilon$ ranges from $0.1$ to $100$, as shown in Figure~\ref{fig:2d-variable-sigma-coeff}.
The computation uses $N_x=N_y=128$, the $S_{32}$ angular rule with $N_\Omega=2048$, and rank $r=32$.
We choose the time step $\Delta t=0.01\Delta x$.
Figure~\ref{fig:2d-variable-sigma-density} shows that SL-DLR remains stable with this large time step, and the density profiles along $y=0$ show good agreement with the SL solution.
Table~\ref{tab:2d-variable-sigma-time} reports the cost reduction from the low-rank representation and the additional gain from angular sampling.

\begin{figure}[htbp!]
    \centering
    \begin{subfigure}[b]{0.3\textwidth}
        \centering
        \includegraphics[width=\textwidth]{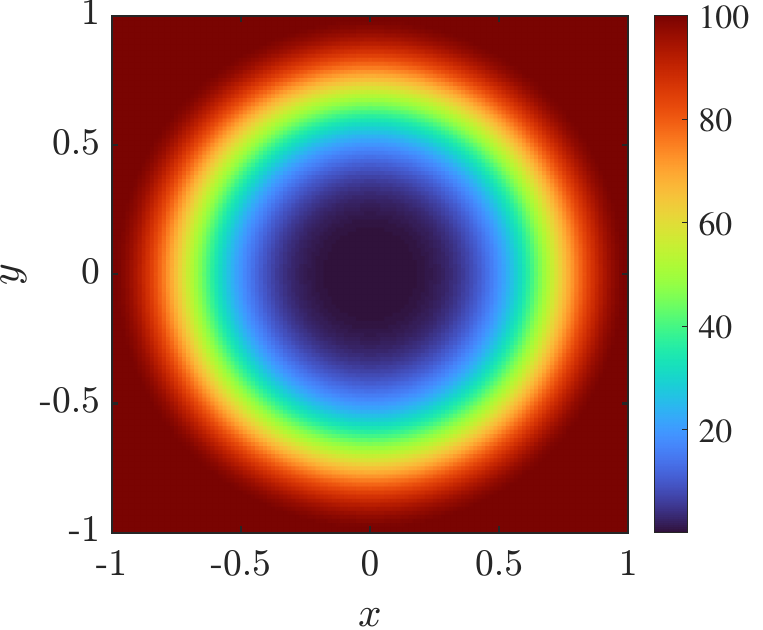}
    \end{subfigure}
    \begin{subfigure}[b]{0.35\textwidth}
        \centering
        \includegraphics[width=\textwidth]{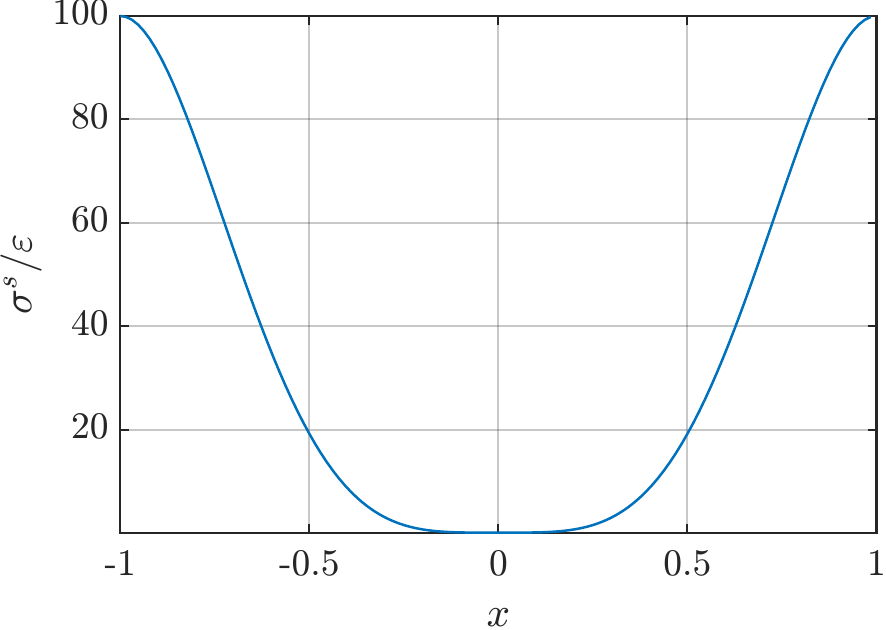}
    \end{subfigure}
    \caption{Example~\ref{ex:2d-variable-sigma}.
        Varying scattering coefficient.
        Plot of $\sigma^s/\varepsilon$ (left) and its profile along $y=0$ (right).
    }
    \label{fig:2d-variable-sigma-coeff}
\end{figure}

\begin{figure}[htbp!]
    \centering
    \begin{subfigure}[c]{0.29\textwidth}
        \centering
        \includegraphics[width=\textwidth]{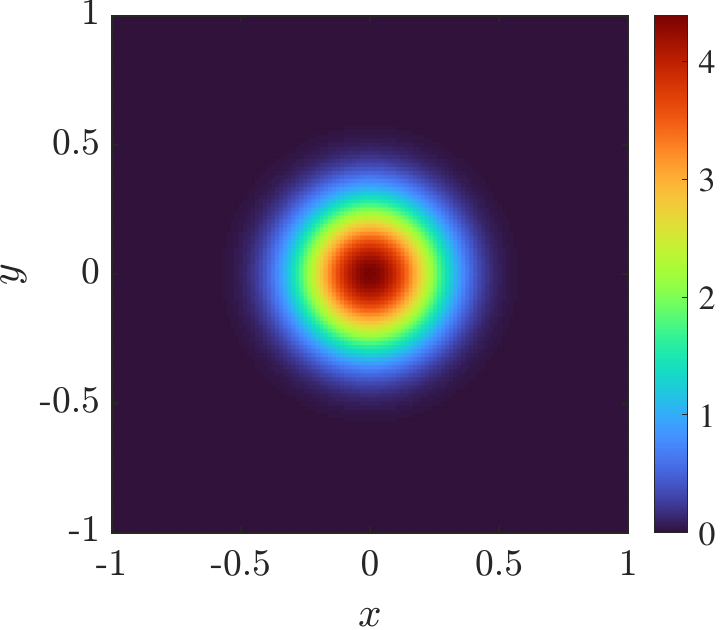}
    \end{subfigure}
    \begin{subfigure}[c]{0.30\textwidth}
        \centering
        \includegraphics[width=\textwidth]{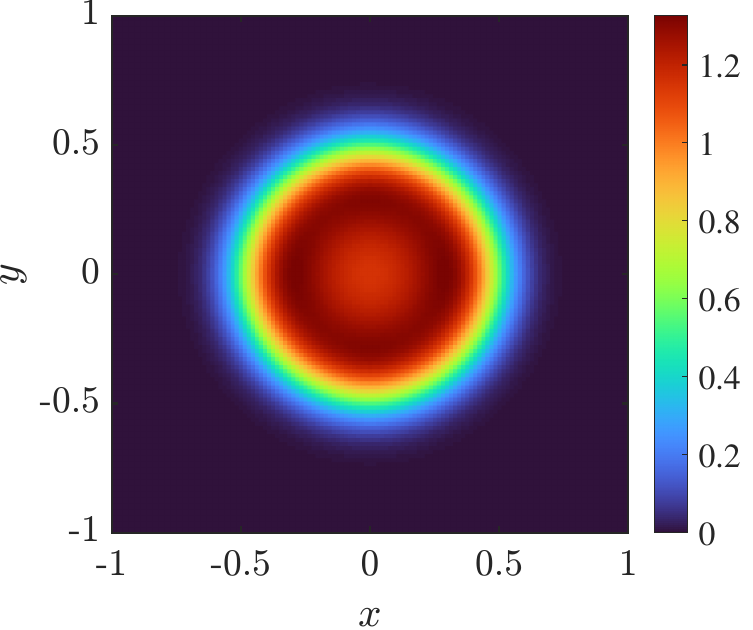}
    \end{subfigure}
    \begin{subfigure}[c]{0.30\textwidth}
        \centering
        \includegraphics[width=\textwidth]{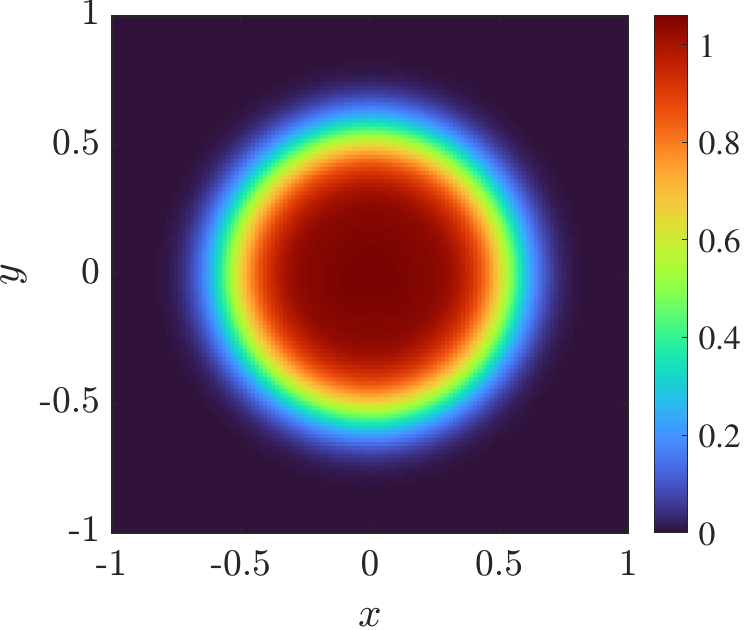}
    \end{subfigure}

    \begin{subfigure}[c]{0.28\textwidth}
        \centering
        \includegraphics[width=\textwidth]{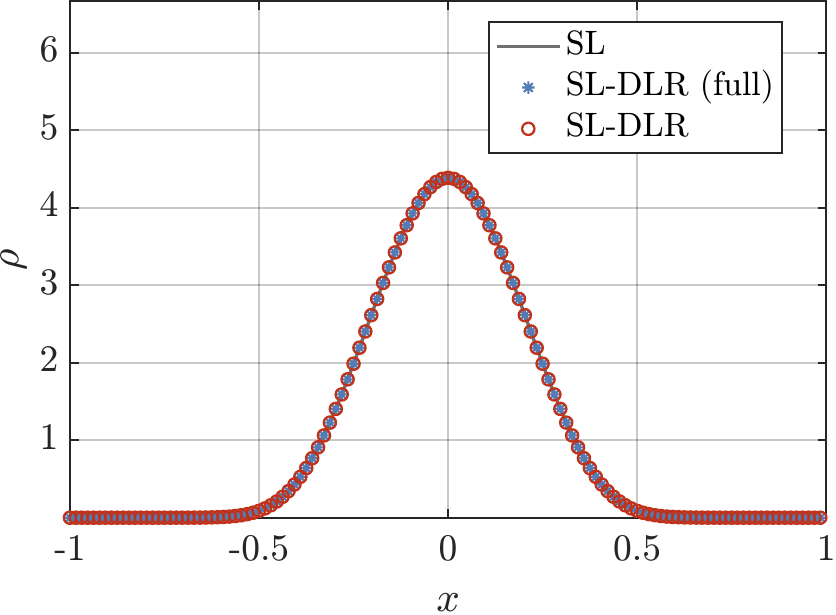}
    \end{subfigure}
    \begin{subfigure}[c]{0.30\textwidth}
        \centering
        \includegraphics[width=\textwidth]{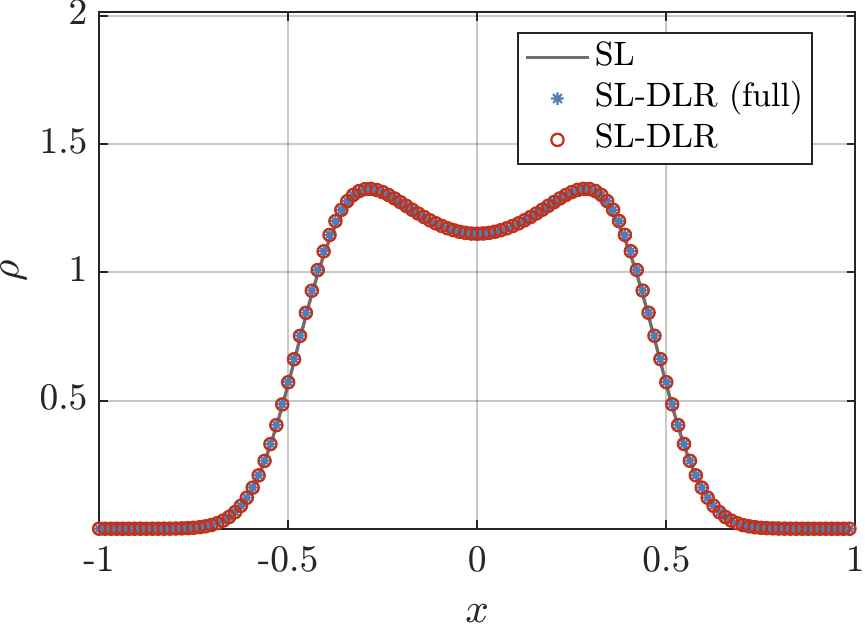}
    \end{subfigure}
    \begin{subfigure}[c]{0.30\textwidth}
        \centering
        \includegraphics[width=\textwidth]{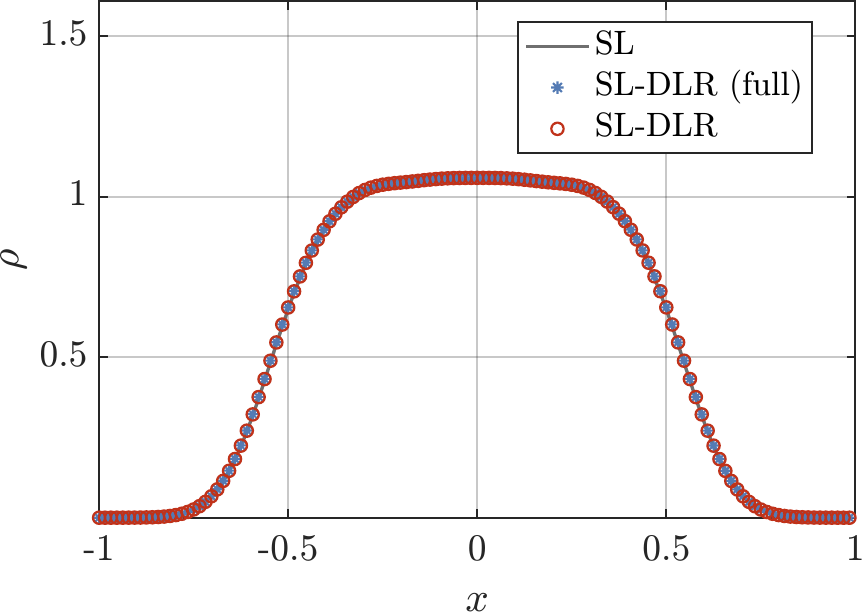}
    \end{subfigure}
    \caption{Example~\ref{ex:2d-variable-sigma}.
        SL-DLR density contours (top row) and density profiles along $y=0$ (bottom row).
        From left to right: $T=0.002$, $T=0.006$, $T=0.010$.}
    \label{fig:2d-variable-sigma-density}
\end{figure}

\begin{table}[htbp!]
    \centering
    \caption{Example~\ref{ex:2d-variable-sigma}.
        Wall-clock cost for $T=0.010$.}
    \label{tab:2d-variable-sigma-time}
    \begin{tabular}{ccccc}
        \toprule
        Method       & Total (s) & Per step (s) & Steps & $\Delta t$ \\
        \midrule
        SL           & 2413.51   & 3.77e+01     & 64    & 1.56e-04   \\
        SL-DLR(full) & 117.99    & 1.84e+00     & 64    & 1.56e-04   \\
        SL-DLR       & 22.21     & 3.47e-01     & 64    & 1.56e-04   \\
        \bottomrule
    \end{tabular}
\end{table}

We use this example to compare SL-DLR with the IMEX-S-BUG method \cite{liJiangZhangXiong2026temporalStability}, a low-rank AP-DLR scheme based on a macro--micro decomposition, an IMEX time discretization, and a Schur-complement formulation.
For SL-DLR, we use two time steps, $\Delta t = 0.01\Delta x$ and $\Delta t = 0.005\Delta x$.
For IMEX-S-BUG, the time step is chosen according to the stability criterion in \cite{liJiangZhangXiong2026temporalStability}, which gives $\Delta t \approx 0.005\Delta x$ for this test.
Figure~\ref{fig:2d-variable-sigma-imex-slice} shows that the two methods give comparable density profiles along $y=0$.
Table~\ref{tab:2d-variable-sigma-imex-time} shows that their per-step costs are of the same order.
Since SL-DLR allows a larger time step in this test, it requires fewer steps and therefore reduces the total wall time.

\begin{figure}[htbp!]
    \centering
    \includegraphics[width=.5\textwidth]{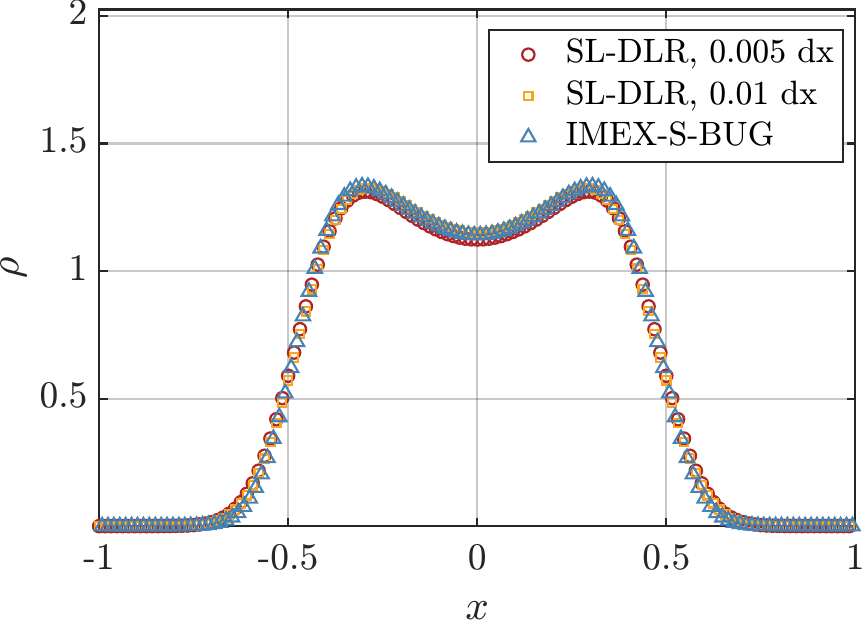}
    \caption{Example~\ref{ex:2d-variable-sigma}.
        Density profiles along $y=0$ for SL-DLR and IMEX-S-BUG at $T=0.006$.}
    \label{fig:2d-variable-sigma-imex-slice}
\end{figure}

\begin{table}[htbp!]
    \centering
    \caption{Example~\ref{ex:2d-variable-sigma}.
        Wall-clock cost comparison with IMEX-S-BUG at $T=0.006$.}
    \label{tab:2d-variable-sigma-imex-time}
    \begin{tabular}{ccccc}
        \toprule
        Method                           & Total (s) & Per step (s) & Steps & $\Delta t$ \\
        \midrule
        SL-DLR, $\Delta t=0.005\Delta x$ & 22.06     & 2.86e-01     & 77    & 7.81e-05   \\
        SL-DLR, $\Delta t=0.01\Delta x$  & 14.26     & 3.66e-01     & 39    & 1.56e-04   \\
        IMEX-S-BUG                       & 51.67     & 6.71e-01     & 77    & 7.82e-05   \\
        \bottomrule
    \end{tabular}
\end{table}

\subsection{Three-dimensional Gaussian data with varying scattering}\label{ex:3d-variable-sigma}

We finally consider a three-dimensional counterpart of Example~\ref{ex:2d-variable-sigma}.
The domain is $[-1,1]^3$, and the initial condition is
\[
    f(0,x,y,z,\bm{\Omega})
    =
    \frac{1}{4\pi \Sigma^2}
    \exp\left(-\frac{x^2+y^2+z^2}{4\Sigma^2}\right),
    \quad
    \Sigma^2 = 10^{-2}.
\]
The scattering coefficient is
\[
    \sigma^s(x,y,z)
    =
    \begin{cases}
        0.999\,c^4(c+\sqrt{2})^2(c-\sqrt{2})^2+0.001,
           & c=\sqrt{x^2+y^2+z^2}<1, \\
        1, & \text{otherwise}.
    \end{cases}
\]
We take $\varepsilon=10^{-2}$.
The discretization uses $N_x=N_y=N_z=48$, the $S_{32}$ angular rule with $N_\Omega=2048$, and rank $r=64$.
We choose the time step $\Delta t=0.01\Delta x$.
Figure~\ref{fig:3d-variable-sigma-density} shows close agreement between the sampled low-rank result and the full-rank SL reference in the center-line profiles.
Table~\ref{tab:3d-variable-sigma-time} reports the efficiency improvement in this three-dimensional test.

\begin{figure}[htbp!]
    \centering
    \begin{subfigure}[c]{0.32\textwidth}
        \centering
        \includegraphics[width=\textwidth]{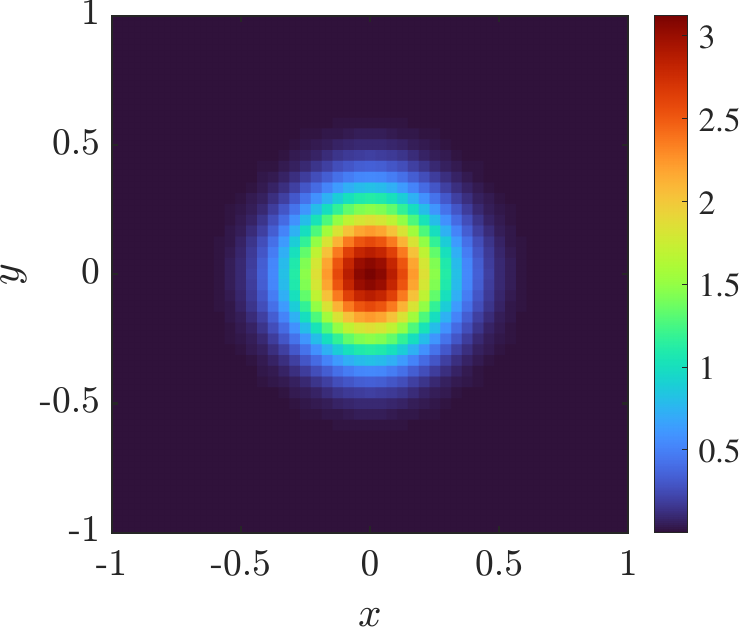}
    \end{subfigure}
    \begin{subfigure}[c]{0.32\textwidth}
        \centering
        \includegraphics[width=\textwidth]{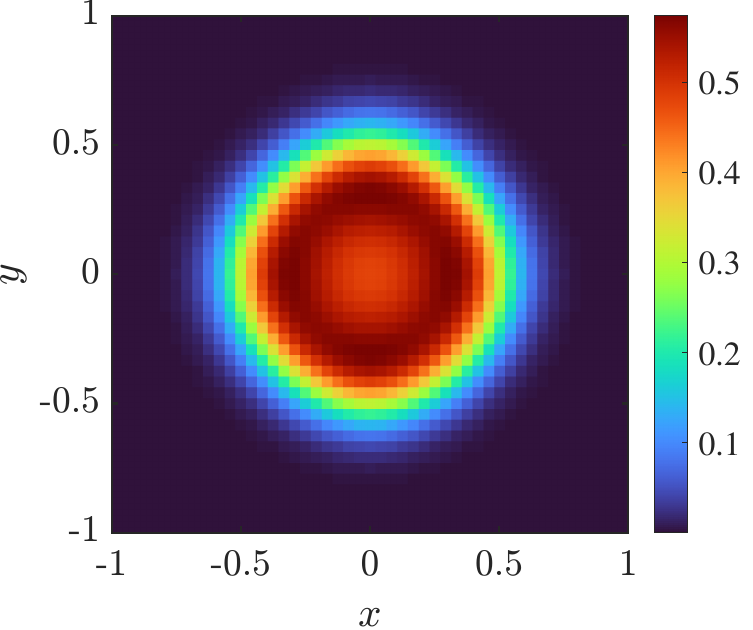}
    \end{subfigure}
    \begin{subfigure}[c]{0.32\textwidth}
        \centering
        \includegraphics[width=\textwidth]{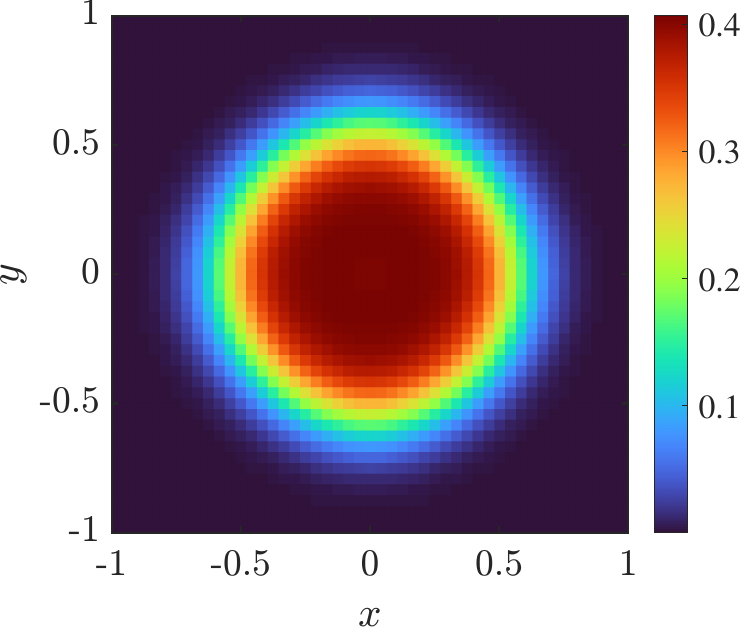}
    \end{subfigure}

    \begin{subfigure}[c]{0.31\textwidth}
        \centering
        \includegraphics[width=\textwidth]{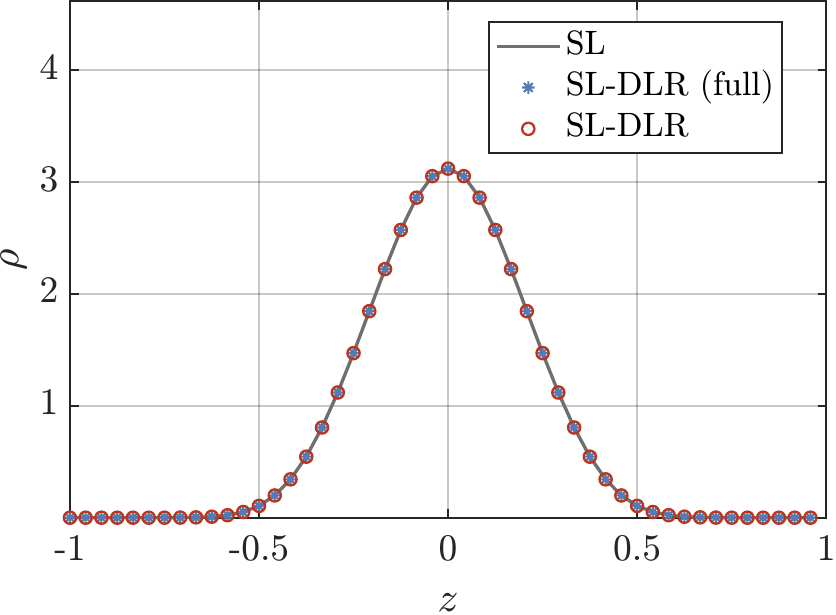}
    \end{subfigure}
    \begin{subfigure}[c]{0.32\textwidth}
        \centering
        \includegraphics[width=\textwidth]{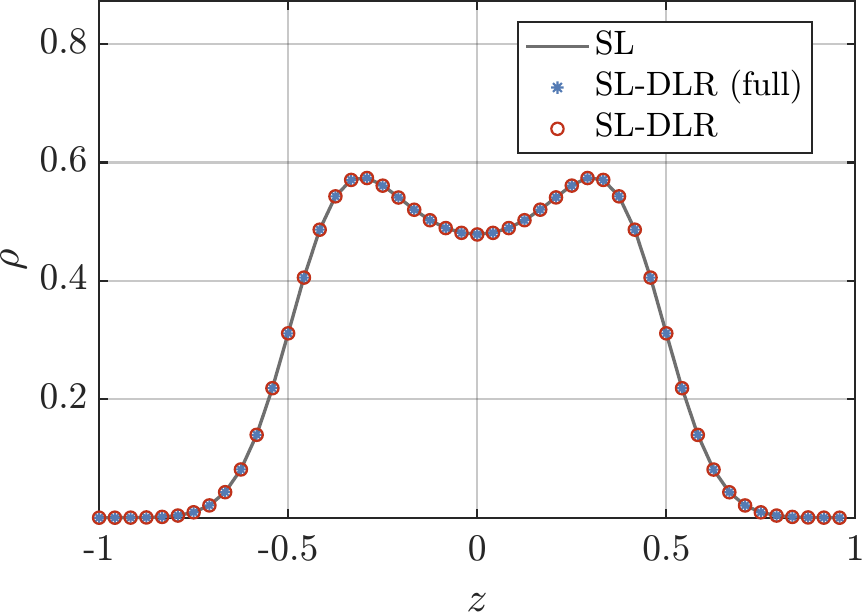}
    \end{subfigure}
    \begin{subfigure}[c]{0.32\textwidth}
        \centering
        \includegraphics[width=\textwidth]{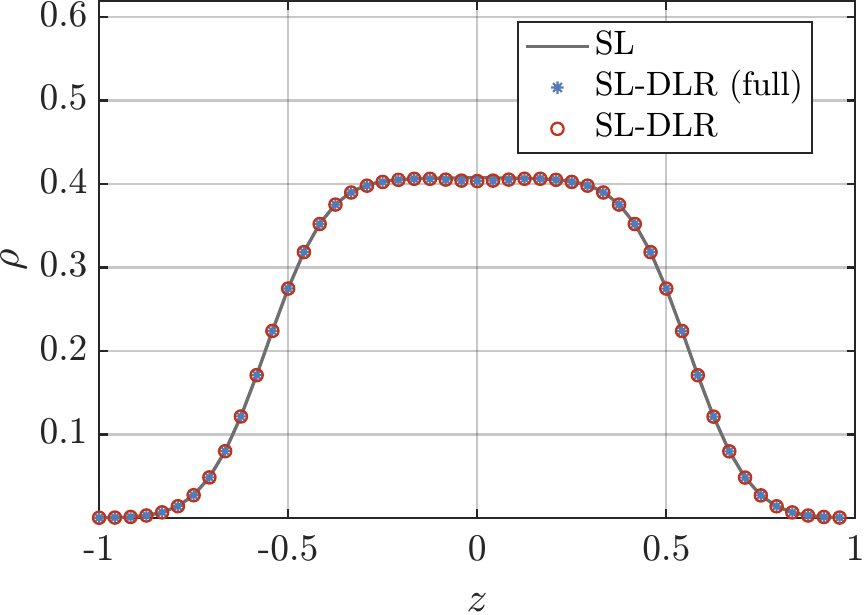}
    \end{subfigure}
    \caption{Example~\ref{ex:3d-variable-sigma}.
        SL-DLR density slices at $z=0$ (top row) and center-line profiles along $x=y=0$ (bottom row).
        From left to right: $T=0.002$, $T=0.006$, $T=0.010$.}
    \label{fig:3d-variable-sigma-density}
\end{figure}

\begin{table}[htbp!]
    \centering
    \caption{Example~\ref{ex:3d-variable-sigma}.
        Wall-clock cost for $T=0.010$.}
    \label{tab:3d-variable-sigma-time}
    \begin{tabular}{ccccc}
        \toprule
        Method       & Total (s) & Per step (s) & Steps & $\Delta t$ \\
        \midrule
        SL           & 22709.27  & 3.03e+02     & 75    & 1.33e-04   \\
        SL-DLR(full) & 1748.40   & 2.33e+01     & 75    & 1.33e-04   \\
        SL-DLR       & 476.76    & 6.36e+00     & 75    & 1.33e-04   \\
        \bottomrule
    \end{tabular}
\end{table}

\begin{remark}
    We use this test to quantify the storage and runtime gains of SL-DLR.
    With $N_s=48^3$ and $N_\Omega=2\times32^2$, the full-rank distribution contains $N_sN_\Omega\approx 2.3\times10^8$ scalar unknowns, whereas the rank-$64$ representation stores $(N_s+N_\Omega)r+r^2\approx 7.2\times10^6$ scalars, about $32$ times fewer.
    The timing results show a comparable gain: SL-DLR is about $29$ times faster than SL, with a factor of about $9$ from low rank and an additional factor of about $3.2$ from angular sampling.
    Although full FD-SL can be parallelized over angular directions, SL-DLR substantially reduces the storage and work requirements for large-scale 3D2V computations.
\end{remark}

\section{Conclusion} \label{sec:conclusion}

We have developed an AP SL-DLR framework for multiscale linear kinetic transport equations.
The method combines the large-time-step capability of SL discretizations with the storage and per-step cost reduction provided by DLR representations.
It couples a macroscopic density update with the BUG integrator for the kinetic distribution, and uses sampled angular quadrature to retain the reduced complexity in the SL flux derivative evaluation.
We establish unconditional stability of the SL-DLR(full) scheme in the constant-coefficient case, and quantify the error induced by angular sampling in the flux derivative.
The AP property of both SL-DLR and SL-DLR(full) is established in the diffusive limit.
Numerical experiments demonstrate that both SL-DLR and SL-DLR(full) are robust across kinetic and diffusive regimes, and confirm the computational efficiency gained from the low-rank representation and sampled flux evaluation.
Future work will focus on extending the proposed framework to nonlinear kinetic equations, adaptive-rank strategies, and high-order semi-Lagrangian discretizations.

\appendix
\section{Energy-stability proof details}\label{app:energy-stability}

We collect the proofs of the two stability propositions from Section~\ref{sec:theoretical-analysis}.
Throughout this appendix, we use the one-dimensional periodic source-free constant-coefficient assumptions stated in Section~\ref{sec:theoretical-analysis}.
Define the nonequilibrium part of the weighted state by $ G^n = Y^n-\rho^n(M\bm{1})^\top.$
Since $\rho^n=Y^nM\bm{1}$ and $\norm*{M\bm{1}}^2 = \sum_j w_j=1$, we have $G^nM\bm{1}=0$.
Hence, the equilibrium and nonequilibrium components are orthogonal in the weighted Frobenius norm, and
\begin{equation}\label{eq:theory-energy-orthogonal-split}
    \norm*{Y^n}_{F,\Delta x}^2 = \norm*{\rho^n}_{\Delta x}^2 + \norm*{G^n}_{F,\Delta x}^2.
\end{equation}

\begin{lemma}\label{lem:app-energy-macro}
    The provisional density $\rho^{n+1,*}$ satisfies
    \begin{equation}\label{eq:theory-energy-macro-tested}
        \begin{aligned}
            \frac12\norm*{\rho^{n+1,*}}_{\Delta x}^2
             & +
            \Delta t\norm*{\sqrt{\beta}D_x^-\rho^{n+1,*}}_{\Delta x}^2
            +
            \Delta t\norm*{\sqrt{\sigma^a}\rho^{n+1,*}}_{\Delta x}^2
            \\
             & \le
            \frac12\norm*{\rho^n}_{\Delta x}^2
            +
            \frac{\Delta t}{\varepsilon}
            \left|\inner*{\alpha_1\mathcal{J}^n,\rho^{n+1,*}}_{\Delta x}\right|.
        \end{aligned}
    \end{equation}
\end{lemma}

\begin{proof}
    In the constant-coefficient source-free case, the macroscopic solve is
    \begin{equation}\label{eq:theory-energy-macro}
        \left(I+\Delta t(-L_\beta+\sigma^a I)\right)\rho^{n+1,*}
        =
        \rho^n-\frac{\Delta t}{\varepsilon}\alpha_1\mathcal{J}^n.
    \end{equation}
    Testing \eqref{eq:theory-energy-macro} by $\rho^{n+1,*}$ in
    $\inner*{\cdot,\cdot}_{\Delta x}$ gives
    \[
        \begin{aligned}
            \norm*{\rho^{n+1,*}}_{\Delta x}^2
             & +
            \Delta t\norm*{\sqrt{\beta}D_x^-\rho^{n+1,*}}_{\Delta x}^2
            +
            \Delta t\norm*{\sqrt{\sigma^a}\rho^{n+1,*}}_{\Delta x}^2
            \\
             & =
            \inner*{\rho^n,\rho^{n+1,*}}_{\Delta x}
            -
            \frac{\Delta t}{\varepsilon}
            \inner*{\alpha_1\mathcal{J}^n,\rho^{n+1,*}}_{\Delta x}.
        \end{aligned}
    \]
    The first right-hand side term is bounded by
    \[
        \inner*{\rho^n,\rho^{n+1,*}}_{\Delta x}
        \le
        \frac12\norm*{\rho^n}_{\Delta x}^2
        +
        \frac12\norm*{\rho^{n+1,*}}_{\Delta x}^2.
    \]
    This gives \eqref{eq:theory-energy-macro-tested}.
\end{proof}

\begin{lemma}\label{lem:app-energy-micro}
    The full-rank microscopic update of the full-rank FD-SL scheme satisfies
    \begin{equation}\label{eq:theory-energy-micro-rhostar}
        \norm*{Y^{n+1}}_{F,\Delta x}^2
        \le
        \frac{1}{(1+\mu\Delta t)^2}
        \Big(
        \norm*{G^n}_{F,\Delta x}^2
        +
        \left(1+\mu^s\Delta t\right)
        \left(
            \norm*{\rho^n}_{\Delta x}^2
            +
            \mu^s\Delta t
            \norm*{\rho^{n+1,*}}_{\Delta x}^2
            \right)
        \Big),
    \end{equation}
    where $\mu^s = \sigma^s / \varepsilon^2$ and $\mu=\mu^s+\sigma^a$.
\end{lemma}

\begin{proof}
    The full-rank microscopic update for the weighted state is
    \begin{equation}\label{eq:theory-energy-micro}
        (1+\mu\Delta t)Y^{n+1}
        +
        \frac{\Delta t}{\varepsilon}
        \left(D_x^-Y^{n+1}Q^+ + D_x^+Y^{n+1}Q^-\right)
        =
        Y^n+
        \mu^s\Delta t
        \rho^{n+1,*}(M\bm{1})^\top.
    \end{equation}
    Taking the inner product of \eqref{eq:theory-energy-micro} with $Y^{n+1}$, and using the nonnegativity of the upwind term,
    \[
        \inner*{
        D_x^-Y^{n+1}Q^+ + D_x^+Y^{n+1}Q^-,
        Y^{n+1}
        }_{F,\Delta x}
        \ge0,
    \]
    gives
    \[
        (1+\mu\Delta t)
        \norm*{Y^{n+1}}_{F,\Delta x}^2
        \le
        \inner*{
            Y^n
            +
            \mu^s\Delta t
            \rho^{n+1,*}(M\bm{1})^\top,
            Y^{n+1}
        }_{F,\Delta x}.
    \]
    Hence, by Cauchy--Schwarz,
    \begin{equation}\label{eq:theory-energy-micro-pre-split}
        \norm*{Y^{n+1}}_{F,\Delta x}^2
        \le
        \frac{1}{(1+\mu\Delta t)^2}
        \norm*{
            Y^n+
            \mu^s\Delta t
            \rho^{n+1,*}(M\bm{1})^\top
        }_{F,\Delta x}^2.
    \end{equation}
    Using $Y^n=\rho^n(M\bm{1})^\top+G^n$ and the orthogonality in
    \eqref{eq:theory-energy-orthogonal-split}, we obtain
    \[
        \norm*{
            Y^n+
            \mu^s\Delta t
            \rho^{n+1,*}(M\bm{1})^\top
        }_{F,\Delta x}^2
        =
        \norm*{G^n}_{F,\Delta x}^2
        +
        \norm*{
            \rho^n+
            \mu^s\Delta t\,\rho^{n+1,*}
        }_{\Delta x}^2.
    \]
    Finally, the Cauchy--Schwarz inequality gives
    \[
        \norm*{
            \rho^n+\mu^s\Delta t\,\rho^{n+1,*}
        }_{\Delta x}^2
        \le
        \left(1+\mu^s\Delta t\right)
        \left(
        \norm*{\rho^n}_{\Delta x}^2
        +
        \mu^s\Delta t
        \norm*{\rho^{n+1,*}}_{\Delta x}^2
        \right).
    \]
    Substitution into \eqref{eq:theory-energy-micro-pre-split} gives \eqref{eq:theory-energy-micro-rhostar} and completes the proof.
\end{proof}

\begin{lemma}\label{lem:app-energy-weak-flux}
    For any periodic test grid function $\phi$, the semi-Lagrangian flux derivative satisfies
    \begin{equation}\label{eq:theory-energy-weak-test}
        \left|\inner*{\mathcal{J}^n,\phi}_{\Delta x}\right|^2
        \le
        \frac{1}{3}\,\norm*{G^n}_{F,\Delta x}^2 \norm*{D_x^-\phi}_{\Delta x}^2.
    \end{equation}
    Moreover, the reconstructed flux derivative has zero discrete mean, namely
    \begin{equation}\label{eq:theory-energy-weak-test-mean}
        \inner*{\mathcal J^n,\bm{1}}_{\Delta x} = \Delta x\sum_i \mathcal J_i^n = 0.
    \end{equation}
\end{lemma}

\begin{proof}
    With $G_j^n=\sqrt{w_j}(f_j^n-\rho^n)$, the macroscopic flux derivative equals
    \[
        \mathcal{J}^n
        =
        \sum_{j=1}^{N_v} w_j v_j \mathcal{B}_j \left(D_x^{\mathrm{up},(j)}(f_j^n-\rho^n)\right)
        =
        \sum_{j=1}^{N_v} \sqrt{w_j}v_j \mathcal{B}_j \left(D_x^{\mathrm{up},(j)}G_j^n\right).
    \]
    Under periodic boundary conditions, the backtracking interpolation operator reads
    \[
        \mathcal{B}_j = (1-\lambda_j) S_{m_j} + \lambda_j S_{m_j+1}, \qquad m_j \in \mathbb{Z},\,\, 0\le \lambda_j<1,
    \]
    where $S_m$ is a periodic shift.
    Hence $\mathcal{B}_j$ is a contraction in $\norm*{\,\cdot\,}_{\Delta x}$,
    \[
        \norm*{\mathcal{B}_j(q)}_{\Delta x}
        \le
        \norm*{q}_{\Delta x},
    \]
    and $\mathcal{B}_j$ commutes with $D_x^\pm$.
    Applying periodic summation by parts and using this commutation property, we obtain
    \[
        \inner*{v_j \mathcal{B}_j \left(D_x^{\mathrm{up},(j)}G_j^n\right), \phi}_{\Delta x}
        =
        - \inner*{v_j \mathcal{B}_j(G_j^n), D_x^{\mathrm{opp},(j)}\phi}_{\Delta x},
    \]
    where $D_x^{\mathrm{opp},(j)}=D_x^+$ if $v_j>0$ and $D_x^{\mathrm{opp},(j)}=D_x^-$ if $v_j\le0$.
    Therefore,
    \[
        \left|\inner*{v_j \mathcal{B}_j \left(D_x^{\mathrm{up},(j)}G_j^n\right), \phi}_{\Delta x}\right|
        \le
        |v_j| \norm*{G_j^n}_{\Delta x}\norm*{D_x^-\phi}_{\Delta x}.
    \]
    Applying Cauchy--Schwarz yields \eqref{eq:theory-energy-weak-test}.
    Taking $\phi = \bm{1}$ gives \eqref{eq:theory-energy-weak-test-mean}.
\end{proof}

\begin{proof}[Proof of Proposition~\ref{prop:theory-fullrank-energy-stability}]
    Let $T=-L_\beta+\sigma^a I$. Then
    \begin{equation}\label{eq:theory-energy-operator-identity}
        \inner*{T\phi,\phi}_{\Delta x}
        =
        \norm*{T^{1/2} \phi}_{\Delta x}^2
        =
        \norm*{\sqrt{\beta} D_x^-\phi}_{\Delta x}^2 + \norm*{\sqrt{\sigma^a} \phi}_{\Delta x}^2.
    \end{equation}
    We estimate the flux-derivative term in \eqref{eq:theory-energy-macro-tested} by
    \begin{align*}
              & \frac{\Delta t}{\varepsilon}
        \left|\inner*{\alpha_1\mathcal{J}^n,\rho^{n+1,*}}_{\Delta x}\right|
        \\
        \le{} &
        \frac{\Delta t \alpha_1^2}{4\varepsilon^2}
        \norm*{T^{-1/2} \mathcal{J}^n}_{\Delta x}^2
        +
        \Delta t \norm*{T^{1/2} \rho^{n+1,*}}_{\Delta x}^2
        \\
        ={}   &
        \frac{\Delta t \alpha_1^2}{4\varepsilon^2}
        \norm*{T^{-1/2} \mathcal{J}^n}_{\Delta x}^2
        + \Delta t \left( \norm*{\sqrt{\beta} D_x^- \rho^{n+1,*}}_{\Delta x}^2 + \norm*{\sqrt{\sigma^a} \rho^{n+1,*}}_{\Delta x}^2\right).
    \end{align*}
    Substituting it into \eqref{eq:theory-energy-macro-tested}, we obtain
    \begin{equation}\label{eq:theory-energy-macro-resolvent}
        \norm*{\rho^{n+1,*}}_{\Delta x}^2
        \le
        \norm*{\rho^n}_{\Delta x}^2
        +
        \frac{\Delta t \alpha_1^2}{2\varepsilon^2}
        \norm*{T^{-1/2} \mathcal{J}^n}_{\Delta x}^2.
    \end{equation}
    If $\sigma^a=0$, $T^{-1/2}$ is understood on the zero-mean subspace; this
    is compatible with Lemma~\ref{lem:app-energy-weak-flux}.

    By Lemma~\ref{lem:app-energy-weak-flux} and \eqref{eq:theory-energy-operator-identity},
    \[
        \left|\inner*{\mathcal{J}^n,\phi}_{\Delta x}\right|^2
        \le
        \frac{1}{3}\,\norm*{G^n}_{F,\Delta x}^2 \norm*{D_x^- \phi}_{\Delta x}^2
        \le
        \frac{1}{3\beta}\,\norm*{G^n}_{F,\Delta x}^2 \norm*{T^{1/2}\phi}_{\Delta x}^2.
    \]
    Taking the dual supremum over $\phi$ yields
    \[
        \norm*{T^{-1/2} \mathcal{J}^n}_{\Delta x}^2
        \le
        \frac{1}{3\beta}\norm*{G^n}_{F,\Delta x}^2.
    \]
    Combining this bound with \eqref{eq:theory-energy-macro-resolvent}, we obtain
    \begin{equation}\label{eq:theory-energy-macro-closed}
        \norm*{\rho^{n+1,*}}_{\Delta x}^2
        \le
        \norm*{\rho^n}_{\Delta x}^2
        +
        \frac{\Delta t}{2\varepsilon^2}
        \frac{\alpha_1^2}{3\beta}\norm*{G^n}_{F,\Delta x}^2.
    \end{equation}
    Substituting \eqref{eq:theory-energy-macro-closed} into \eqref{eq:theory-energy-micro-rhostar} gives
    \begin{equation}\label{eq:theory-energy-two-coefficients}
        \norm*{Y^{n+1}}_{F,\Delta x}^2
        \le
        \Gamma_\rho\norm*{\rho^n}_{\Delta x}^2
        +
        \Gamma_G\norm*{G^n}_{F,\Delta x}^2 ,
    \end{equation}
    where
    \[
        \Gamma_\rho
        ={}
        \frac{
            \left(1+\mu^s\Delta t\right)^2}
        {(1+\mu\Delta t)^2},
        \qquad
        \Gamma_G
        ={}
        \frac{
            1+
            \mu^s\Delta t
            \left(1+\mu^s\Delta t\right)
            \frac{\Delta t}{2\varepsilon^2}\frac{\alpha_1^2}{3\beta}
        }{(1+\mu\Delta t)^2}.
    \]
    Since $\mu \ge \mu^s$, we have $\Gamma_\rho\le1$.
    Using the constant-coefficient definitions of $\alpha_1$ and $\beta$,
    \[
        \frac{\Delta t}{\varepsilon^2}\frac{\alpha_1^2}{3\beta}
        =
        \frac{\mu}{\mu^s}
        \psi(\mu\Delta t),
        \qquad
        \psi(z):=\frac{z e^{-z}}{e^{z}-1}.
    \]
    Since $e^z - 1\ge z\ge z e^{-z}$ for $z>0$, we have
    $0<\psi(z)\le1$.  Therefore
    \[
        (1+\mu\Delta t)^2\Gamma_G
        ={}
        1 + \frac12
        \left(1+\mu^s\Delta t\right)
        \mu\Delta t\,\psi(\mu\Delta t)
        \le{}
        1 + \frac12(1+\mu\Delta t)\mu\Delta t
        \le
        (1+\mu\Delta t)^2,
    \]
    and hence $\Gamma_G\le1$.
    The stability estimate follows from \eqref{eq:theory-energy-two-coefficients}
    and \eqref{eq:theory-energy-orthogonal-split}.
\end{proof}

\begin{proof}[Proof of Proposition~\ref{prop:theory-lowrank-full-energy-stability}]
    Let $X^{n+1}$ and $V^{n+1}$ be the bases obtained from the K- and L-steps, respectively, and define $P^X=X^{n+1}(X^{n+1})^\top$, $P^V=V^{n+1}(V^{n+1})^\top$.
    The projected S-step reads
    \begin{equation}\label{eq:theory-energy-lowrank-projected-s}
        \begin{aligned}
             & (1+\mu\Delta t)Y_{\mathrm{SL\text{-}DLR(full)}}^{n+1}
            +
            \frac{\Delta t}{\varepsilon}
            P^X
            \left(
            D_x^-Y_{\mathrm{SL\text{-}DLR(full)}}^{n+1}Q^+
            +
            D_x^+Y_{\mathrm{SL\text{-}DLR(full)}}^{n+1}Q^-
            \right)
            P^V
            \\
             & \qquad =
            P^X
            \left(
            Y^n+
            \mu^s\Delta t
            \rho^{n+1,*}(M\bm{1})^\top
            \right)
            P^V.
        \end{aligned}
    \end{equation}
    Since $Y_{\mathrm{SL\text{-}DLR(full)}}^{n+1}\in \left\{X^{n+1}S(V^{n+1})^\top: S\in\mathbb{R}^{r\times r}\right\}$, for any matrix $A$,
    \[
        \inner*{
            P^XAP^V,
            Y_{\mathrm{SL\text{-}DLR(full)}}^{n+1}
        }_{F,\Delta x}
        =
        \inner*{
            A,
            Y_{\mathrm{SL\text{-}DLR(full)}}^{n+1}
        }_{F,\Delta x}.
    \]
    Testing \eqref{eq:theory-energy-lowrank-projected-s} with $Y_{\mathrm{SL\text{-}DLR(full)}}^{n+1}$ therefore removes the projectors from all inner products.
    The Cauchy--Schwarz inequality then gives
    \begin{align*}
        \norm*{Y_{\mathrm{SL\text{-}DLR(full)}}^{n+1}}_{F,\Delta x}^2
        \le
        \frac{1}{(1+\mu\Delta t)^2}\norm*{Y^n + \mu^s\Delta t \rho^{n+1,*}(M\bm{1})^\top}_{F,\Delta x}^2.
    \end{align*}
    This coincides with \eqref{eq:theory-energy-micro-pre-split}.
    The remainder of the proof follows the same argument as in the proof of Proposition~\ref{prop:theory-fullrank-energy-stability}.
\end{proof}

\bibliographystyle{plain}
\bibliography{references}

@article{an2008optimizing,
  author    = {An, Steven S and Kim, Theodore and James, Doug L},
  title     = {Optimizing cubature for efficient integration of subspace deformations},
  journal   = {ACM transactions on graphics (TOG)},
  year      = {2008},
  number    = {5},
  pages     = {1--10},
  publisher = {ACM New York, NY, USA},
  volume    = {27}
}

@article{bachmayr2023low,
  author  = {Bachmayr, Markus},
  title   = {Low-Rank Tensor Methods for Partial Differential Equations},
  journal = {Acta Numerica},
  year    = {2023},
  pages   = {1--121},
  volume  = {32}
}

@article{boscarino2013implicit,
  author  = {Boscarino, Sebastiano and Pareschi, Lorenzo and Russo, Giovanni},
  title   = {Implicit-Explicit Runge--Kutta Schemes for Hyperbolic Systems and Kinetic Equations in the Diffusion Limit},
  journal = {SIAM Journal on Scientific Computing},
  year    = {2013},
  number  = {1},
  pages   = {A22--A51},
  volume  = {35}
}

@article{cai2024asymptotic,
  author  = {Cai, Yi and Zhang, Guoliang and Zhu, Hongqiang and Xiong, Tao},
  title   = {Asymptotic Preserving Semi-{{Lagrangian}} Discontinuous {{Galerkin}} Methods for Multiscale Kinetic Transport Equations},
  journal = {Journal of Computational Physics},
  year    = 2024,
  doi     = {10.1016/j.jcp.2024.113190},
  issn    = {00219991},
  pages   = {113190},
  volume  = {513}
}

@book{case1967,
  author    = {Case, K. M. and Zweifel, P. F.},
  title     = {Linear Transport Theory},
  year      = {1967},
  publisher = {Addison-Wesley}
}

@book{Cercignani1988,
  author    = {Cercignani, Carlo},
  title     = {The Boltzmann Equation and Its Applications},
  year      = {1988},
  publisher = {Springer},
  series    = {Applied Mathematical Sciences},
  volume    = {67}
}

@article{ceruti2022rank,
  author    = {Ceruti, Gianluca and Kusch, Jonas and Lubich, Christian},
  title     = {A Rank-Adaptive Robust Integrator for Dynamical Low-Rank Approximation},
  journal   = {BIT Numerical Mathematics},
  year      = {2022},
  number    = {4},
  pages     = {1149--1174},
  publisher = {Springer},
  volume    = {62}
}

@article{ceruti2022unconventional,
  author  = {Ceruti, Gianluca and Lubich, Christian},
  title   = {An Unconventional Robust Integrator for Dynamical Low-Rank Approximation},
  journal = {BIT Numerical Mathematics},
  year    = {2022},
  pages   = {23--44},
  volume  = {62}
}

@article{ceruti2025galerkin,
  author  = {Ceruti, Gianluca and Crouseilles, Nicolas and Einkemmer, Lukas},
  title   = {A Galerkin Alternating Projection Method for Kinetic Equations in the Diffusive Limit},
  journal = {arXiv preprint arXiv:2505.19929},
  year    = {2025}
}

@book{Chandrasekhar1960,
  author    = {Chandrasekhar, S.},
  title     = {Radiative Transfer},
  year      = {1960},
  publisher = {Dover Publications}
}

@article{chaturantabut2010nonlinear,
  author    = {Chaturantabut, Saifon and Sorensen, Danny C},
  title     = {Nonlinear model reduction via discrete empirical interpolation},
  journal   = {SIAM Journal on Scientific Computing},
  year      = {2010},
  number    = {5},
  pages     = {2737--2764},
  publisher = {SIAM},
  volume    = {32}
}

@article{dektor2024interpolatoryBGK,
  author  = {Dektor, Alec and Einkemmer, Lukas},
  title   = {Interpolatory Dynamical Low-Rank Approximation for the 3+3d {Boltzmann}-{BGK} Equation},
  journal = {arXiv preprint arXiv:2411.15990},
  year    = {2024}
}

@article{DimarcoPareschi2014,
  author  = {Dimarco, Giacomo and Pareschi, Lorenzo},
  title   = {Numerical Methods for Kinetic Equations},
  journal = {Acta Numerica},
  year    = {2014},
  doi     = {10.1017/S0962492914000063},
  pages   = {369--520},
  volume  = {23}
}

@article{ding2021dynamical,
  author  = {Ding, Zhiyan and Einkemmer, Lukas and Li, Qin},
  title   = {Dynamical Low-Rank Integrator for the Linear Boltzmann Equation: Error Analysis in the Diffusion Limit},
  journal = {SIAM Journal on Numerical Analysis},
  year    = {2021},
  number  = {4},
  pages   = {2254--2285},
  volume  = {59}
}

@article{ding2023accuracy,
  author  = {Ding, Mingchang and Qiu, Jing-Mei and Shu, Ruiwen},
  title   = {Accuracy and Stability Analysis of the Semi-Lagrangian Method for Stiff Hyperbolic Relaxation Systems and Kinetic {BGK} Model},
  journal = {Multiscale Modeling \& Simulation},
  year    = {2023},
  number  = {1},
  pages   = {143--167},
  volume  = {21}
}

@article{drmac2016new,
  author    = {Drmac, Zlatko and Gugercin, Serkan},
  title     = {A new selection operator for the discrete empirical interpolation method---improved a priori error bound and extensions},
  journal   = {SIAM Journal on Scientific Computing},
  year      = {2016},
  number    = {2},
  pages     = {A631--A648},
  publisher = {SIAM},
  volume    = {38}
}

@article{einkemmer2018lowrank,
  author  = {Einkemmer, Lukas and Lubich, Christian},
  title   = {A Low-Rank Projector-Splitting Integrator for the {Vlasov--Poisson} Equation},
  journal = {SIAM Journal on Scientific Computing},
  year    = {2018},
  doi     = {10.1137/18M116383X},
  issn    = {1064-8275, 1095-7197},
  number  = {5},
  pages   = {B1330--B1360},
  volume  = {40}
}

@article{einkemmer2021efficient,
  author  = {Einkemmer, Lukas and Hu, Jingwei and Ying, Lexing},
  title   = {An Efficient Dynamical Low-Rank Algorithm for the {Boltzmann}-{BGK} Equation Close to the Compressible Viscous Flow Regime},
  journal = {SIAM Journal on Scientific Computing},
  year    = {2021},
  number  = {5},
  pages   = {B1057--B1080},
  volume  = {43}
}

@article{einkemmer2025review,
  author  = {Einkemmer, Lukas and Kormann, Katharina and Kusch, Jonas and McClarren, Ryan G. and Qiu, Jing-Mei},
  title   = {A Review of Low-Rank Methods for Time-Dependent Kinetic Simulations},
  journal = {Journal of Computational Physics},
  year    = {2025},
  doi     = {10.1016/j.jcp.2025.114191},
  issn    = {00219991},
  pages   = {114191},
  volume  = {538}
}

@article{einkemmerHuKusch2024energyStable,
  author  = {Einkemmer, Lukas and Hu, Jingwei and Kusch, Jonas},
  title   = {Asymptotic-Preserving and Energy Stable Dynamical Low-Rank Approximation},
  journal = {SIAM Journal on Numerical Analysis},
  year    = {2024},
  doi     = {10.1137/23M1547603},
  number  = {1},
  pages   = {73--92},
  volume  = {62}
}

@article{einkemmerHuWang2021apDLR,
  author  = {Einkemmer, Lukas and Hu, Jingwei and Wang, Yubo},
  title   = {An Asymptotic-Preserving Dynamical Low-Rank Method for the Multi-Scale Multi-Dimensional Linear Transport Equation},
  journal = {Journal of Computational Physics},
  year    = {2021},
  doi     = {10.1016/j.jcp.2021.110353},
  pages   = {110353},
  volume  = {439}
}

@article{frank2025asymptotic,
  author  = {Frank, Martin and Kusch, Jonas and Patwardhan, Chinmay},
  title   = {Asymptotic-Preserving and Energy Stable Dynamical Low-Rank Approximation for Thermal Radiative Transfer Equations},
  journal = {Multiscale Modeling \& Simulation},
  year    = {2025},
  doi     = {10.1137/24M1646303},
  number  = {1},
  pages   = {278--312},
  volume  = {23}
}

@book{ganapol2008analytical,
  author    = {Ganapol, Barry D.},
  title     = {Analytical Benchmarks for Nuclear Engineering Applications: Case Studies in Neutron Transport Theory},
  year      = {2008},
  address   = {Issy-les-Moulineaux, France},
  isbn      = {978-92-64-99056-2},
  note      = {NEA No. 6292},
  number    = {NEA/DB/DOC(2008)1},
  publisher = {Nuclear Energy Agency, Organisation for Economic Co-operation and Development},
  series    = {OECD/NEA Data Bank}
}

@article{Ghahremani2024,
  author  = {Ghahremani, Behzad and Babaee, Hessam},
  title   = {A {DEIM} {Tucker} Tensor Cross Algorithm and Its Application to Dynamical Low-Rank Approximation},
  journal = {Computer Methods in Applied Mechanics and Engineering},
  year    = {2024},
  doi     = {10.1016/j.cma.2024.116879},
  pages   = {116879},
  volume  = {423}
}

@article{guo2024conservative,
  author    = {Guo, Wei and Qiu, Jing-Mei},
  title     = {A Conservative Low Rank Tensor Method for the {Vlasov} Dynamics},
  journal   = {SIAM Journal on Scientific Computing},
  year      = {2024},
  doi       = {10.1137/22M1473960},
  issn      = {1064-8275},
  number    = {1},
  pages     = {A232--A263},
  publisher = {Society for Industrial and Applied Mathematics},
  volume    = {46}
}

@article{guo2024local,
  author    = {Guo, Wei and Ema, Jannatul Ferdous and Qiu, Jing-Mei},
  title     = {A Local Macroscopic Conservative ({LoMaC}) Low Rank Tensor Method with the Discontinuous Galerkin Method for the {Vlasov} Dynamics},
  journal   = {Communications on Applied Mathematics and Computation},
  year      = {2024},
  number    = {1},
  pages     = {550--575},
  publisher = {Springer},
  volume    = {6}
}

@article{guo2025inexact,
  author  = {Guo, Wei and Peng, Zhichao},
  title   = {An Inexact Low-Rank Source Iteration for Steady-State Radiative Transfer Equation with Diffusion Synthetic Acceleration},
  journal = {arXiv preprint arXiv:2509.00805},
  year    = {2025}
}

@article{guo2026highly,
  author  = {Guo, Wei and Peng, Zhichao},
  title   = {Highly Efficient Rank-Adaptive Sweep-based {{SI-DSA}} for the Radiative Transfer Equation via Mild Space Augmentation},
  journal = {arXiv preprint arXiv:2603.25233},
  year    = {2026}
}

@article{haut2026efficient,
  author  = {Haut, Terry and Loffeld, John and Einkemmer, Lukas and Guthrey, Pierson and Brunner, Stefan and Schill, William},
  title   = {Efficient {{SN-like}} and {{PN-like}} Dynamic Low Rank methods for Thermal Radiative Transfer},
  journal = {arXiv preprint arXiv:2601.18705},
  year    = {2026}
}

@article{hernandez2017dimensional,
  author    = {Hernandez, Joaquin Alberto and Caicedo, Manuel Alejandro and Ferrer, Alex},
  title     = {Dimensional hyper-reduction of nonlinear finite element models via empirical cubature},
  journal   = {Computer methods in applied mechanics and engineering},
  year      = {2017},
  pages     = {687--722},
  publisher = {Elsevier},
  volume    = {313}
}

@article{jin1999efficient,
  author  = {Jin, Shi},
  title   = {Efficient Asymptotic-Preserving ({AP}) Schemes for Some Multiscale Kinetic Equations},
  journal = {SIAM Journal on Scientific Computing},
  year    = {1999},
  pages   = {441--454},
  volume  = {21}
}

@article{jin2022asymptotic,
  author  = {Jin, Shi},
  title   = {Asymptotic-Preserving Schemes for Multiscale Physical Problems},
  journal = {Acta Numerica},
  year    = {2022},
  pages   = {1--82},
  volume  = {31}
}

@article{JLQX15,
  author  = {Jang, J. and Li, F. and Qiu, J.-M. and Xiong, T.},
  title   = {High Order Asymptotic Preserving {DG}-{IMEX} Schemes for Discrete-Velocity Kinetic Equations in a Diffusive Scaling},
  journal = {Journal of Computational Physics},
  year    = {2015},
  pages   = {199--224},
  volume  = {281}
}

@article{kieri2016discretized,
  author  = {Kieri, Emil and Lubich, Christian and Walach, Hanna},
  title   = {Discretized Dynamical Low-Rank Approximation in the Presence of Small Singular Values},
  journal = {SIAM Journal on Numerical Analysis},
  year    = {2016},
  doi     = {10.1137/15M1026791},
  pages   = {1020--1038},
  volume  = {54}
}

@article{koch2007dynamical,
  author    = {Koch, Othmar and Lubich, Christian},
  title     = {Dynamical Low-Rank Approximation},
  journal   = {SIAM Journal on Matrix Analysis and Applications},
  year      = {2007},
  number    = {2},
  pages     = {434--454},
  publisher = {SIAM},
  volume    = {29}
}

@article{Kormann2015,
  author  = {Kormann, Katharina},
  title   = {A Semi-Lagrangian {Vlasov} Solver in Tensor Train Format},
  journal = {SIAM Journal on Scientific Computing},
  year    = {2015},
  doi     = {10.1137/140971270},
  number  = {4},
  pages   = {B613--B632},
  volume  = {37}
}

@article{kusch2022low,
  author  = {Kusch, Jonas and Whewell, Benjamin and McClarren, Ryan G. and Frank, Martin},
  title   = {A Low-Rank Power Iteration Scheme for Neutron Transport Criticality Problems},
  journal = {Journal of Computational Physics},
  year    = {2022},
  pages   = {111587},
  volume  = {470}
}

@article{kusch2023robust,
  author  = {Kusch, Jonas and Stammer, Pia},
  title   = {A Robust Collision Source Method for Rank Adaptive Dynamical Low-Rank Approximation in Radiation Therapy},
  journal = {ESAIM: Mathematical Modelling and Numerical Analysis},
  year    = {2023},
  number  = {2},
  pages   = {865--891},
  volume  = {57}
}

@article{larsen2009advances,
  author  = {Larsen, Edward W. and Morel, Jim E.},
  title   = {Advances in Discrete-Ordinates Methodology},
  journal = {Nuclear Computational Science: A Century in Review},
  year    = {2009},
  pages   = {1--84}
}

@article{lemou2008new,
  author  = {Lemou, Mohammed and Mieussens, Luc},
  title   = {A New Asymptotic Preserving Scheme Based on Micro-Macro Formulation for Linear Kinetic Equations in the Diffusion Limit},
  journal = {SIAM Journal on Scientific Computing},
  year    = {2008},
  pages   = {334--368},
  volume  = {31}
}

@book{lewis1983computational,
  author    = {Lewis, Elmer Eugene and Miller, Warren F},
  title     = {{Computational Methods of Neutron Transport}},
  year      = {1983},
  publisher = {John Wiley and Sons, Inc., New York, NY}
}

@article{liJiangZhangXiong2026temporalStability,
  author  = {Li, Shun and Jiang, Yan and Zhang, Mengping and Xiong, Tao},
  title   = {Temporal-Stability-Enhanced and Energy-Stable Dynamical Low-Rank Approximation for Multiscale Linear Kinetic Transport Equations},
  journal = {arXiv preprint arXiv:2602.12337},
  year    = {2026},
  url     = {https://arxiv.org/abs/2602.12337}
}

@article{liu2010analysis,
  author  = {Liu, Jian-Guo and Mieussens, Luc},
  title   = {Analysis of an Asymptotic Preserving Scheme for Linear Kinetic Equations in the Diffusion Limit},
  journal = {SIAM Journal on Numerical Analysis},
  year    = {2010},
  pages   = {1474--1491},
  volume  = {48}
}

@article{lubich2014projector,
  author  = {Lubich, Christian and Oseledets, Ivan V.},
  title   = {A Projector-Splitting Integrator for Dynamical Low-Rank Approximation},
  journal = {BIT Numerical Mathematics},
  year    = {2014},
  doi     = {10.1007/s10543-013-0454-0},
  pages   = {171--188},
  volume  = {54}
}

@article{peng2020low,
  author  = {Peng, Zhuogang and McClarren, Ryan G. and Frank, Martin},
  title   = {A Low-Rank Method for Two-Dimensional Time-Dependent Radiation Transport Calculations},
  journal = {Journal of Computational Physics},
  year    = {2020},
  pages   = {109735},
  volume  = {421}
}

@article{peng2021asymptotic,
  author  = {Peng, Zhichao and Li, Fengyan},
  title   = {Asymptotic Preserving {IMEX-DG-S} Schemes for Linear Kinetic Transport Equations Based on Schur Complement},
  journal = {SIAM Journal on Scientific Computing},
  year    = {2021},
  doi     = {10.1137/20M134486X},
  number  = {2},
  pages   = {A1194--A1220},
  volume  = {43}
}

@article{peng2023sweep,
  author  = {Peng, Zhuogang and McClarren, Ryan G.},
  title   = {A Sweep-Based Low-Rank Method for the Discrete Ordinate Transport Equation},
  journal = {Journal of Computational Physics},
  year    = {2023},
  pages   = {111748},
  volume  = {473}
}

@article{sands2025high,
  author    = {Sands, William A and Guo, Wei and Qiu, Jing-Mei and Xiong, Tao},
  title     = {High-order adaptive rank integrators for multiscale linear kinetic transport equations in the hierarchical tucker format},
  journal   = {SIAM Journal on Scientific Computing},
  year      = {2025},
  number    = {6},
  pages     = {A3383--A3412},
  publisher = {SIAM},
  volume    = {47}
}

@article{sandsQiuHayesZheng2026greedyBGK,
  author  = {Sands, William A. and Qiu, Jing-Mei and Hayes, Daniel and Zheng, Nanyi},
  title   = {An Adaptive-Rank Approach with Greedy Sampling for Multi-Scale {BGK} Equations},
  journal = {Journal of Computational Physics},
  year    = {2026},
  doi     = {10.1016/j.jcp.2025.114523},
  pages   = {114523},
  volume  = {547}
}

@article{tyrtyshnikov1995pseudo,
  author  = {Tyrtyshnikov, Eugene E. and Goreinov, Sergei A. and Zamarashkin, Nikolai L.},
  title   = {Pseudo-Skeleton Approximations},
  journal = {Doklady Akademii Nauk},
  year    = {1995},
  number  = {2},
  pages   = {151--152},
  volume  = {343}
}

@article{zhang2023asymptotic,
  author  = {Zhang, Guoliang and Zhu, Hongqiang and Xiong, Tao},
  title   = {Asymptotic {{Preserving}} and {{Uniformly Unconditionally Stable Finite Difference Schemes}} for {{Kinetic Transport Equations}}},
  journal = {SIAM Journal on Scientific Computing},
  year    = 2023,
  doi     = {10.1137/22M1533815},
  issn    = {1064-8275, 1095-7197},
  number  = {5},
  pages   = {B697-B730},
  volume  = {45}
}

@article{zheng2025semi,
  author  = {Zheng, Nanyi and Hayes, Daniel and Christlieb, Andrew and Qiu, Jing-Mei},
  title   = {A Semi-Lagrangian Adaptive-Rank ({SLAR}) Method for Linear Advection and Nonlinear {Vlasov}-{Poisson} System},
  journal = {Journal of Computational Physics},
  year    = {2025},
  pages   = {113970},
  volume  = {532}
}

@article{zheng2025highDimensionalSLAR,
  author  = {Zheng, Nanyi and Sands, William A. and Hayes, Daniel and Christlieb, Andrew J. and Qiu, Jing-Mei},
  title   = {A Semi-Lagrangian Adaptive Rank ({SLAR}) Method for High-Dimensional {Vlasov} Dynamics},
  journal = {arXiv preprint arXiv:2510.24861},
  year    = {2025},
  url     = {https://arxiv.org/abs/2510.24861}
}

\end{document}